\documentclass[notitlepage]{amsart}

\usepackage{amsfonts}
\usepackage[latin1]{inputenc}
\newtheorem{pro}{Proposition}[section]
\newtheorem{teo}[pro]{Theorem}
\newtheorem{defi}[pro]{Definition}
\newtheorem{lem}[pro]{Lemma}
\newtheorem{cor}[pro]{Corollary}
\newtheorem{remark}[pro]{Remark}

\newcommand{\pd}{{\mathrm{pd}}}
\newcommand{\Gpd}{{\mathrm{Gpd}}}
\newcommand{\GPd}{{\mathrm{GPd}}}

\newcommand{\modu}{{\mathrm{mod}}}
\newcommand{\Mod}{{\mathrm{Mod}}}
\newcommand{\ind}{{\mathrm{ind}}}
\newcommand{\Ima}{{\mathrm{Im}}}
\newcommand{\Ker}{{\mathrm{Ker}}}
\newcommand{\Coker}{{\mathrm{Coker}}}
\newcommand{\proj}{{\mathrm{proj}}}
\newcommand{\Proj}{{\mathrm{Proj}}}
\newcommand{\Gproj}{{\mathrm{Gproj}}}
\newcommand{\GProj}{{\mathrm{GProj}}}
\newcommand{\Ginj}{{\mathrm{Ginj}}}

\newcommand{\inj}{{\mathrm{inj}}}

\newcommand{\id}{{\mathrm{id}}}
\newcommand{\Gid}{{\mathrm{Gid}}}
\newcommand{\fid}{{\mathrm{fid}}}
\newcommand{\pdF}{{\mathrm{pd}_F}}
\newcommand{\findim}{{\mathrm{fin.dim}}}
\newcommand{\fpd}{{\mathrm{fpd}}}
\newcommand{\finid}{{\mathrm{fid}}}
\newcommand{\Gfindim}{{\mathrm{Gfin.dim}}}
\newcommand{\Ggldim}{{\mathrm{Ggl.dim}}}
\newcommand{\resdim}{{\mathrm{resdim}}}
\newcommand{\tresdim}{{\mathrm{t.resdim}}}
\newcommand{\repdim}{{\mathrm{rep.dim}}}
\newcommand{\coresdim}{{\mathrm{coresdim}}}
\newcommand{\gldim}{{\mathrm{gl.dim}}}
\newcommand{\Hom}{{\mathrm{Hom}}}
\newcommand{\Obj}{{\mathrm{Obj}}}
\newcommand{\uHom}{\underline{{\mathrm{Hom}}}}

\newcommand{\End}{{\mathrm{End}}}
\newcommand{\uEnd}{\underline{{\mathrm{End}}}}
\newcommand{\Ext}{{\mathrm{Ext}}}
\newcommand{\Phidim}{{\Phi\,\mathrm{dim}}}
\newcommand{\Psidim}{{\Psi\,\mathrm{dim}}}

\newcommand{\Fixdim}{{\Phi_{\X}\mathrm{dim}}}
\newcommand{\Fiedim}{{\Phi_{\E}\mathrm{dim}}}

\newcommand{\Psix}{{\Psi_{\X}}}
\newcommand{\Psixdim}{{\Psi_{\X}\mathrm{dim}}}
\newcommand{\Psiedim}{{\Psi_{\E}\mathrm{dim}}}
\newcommand{\PsiGpdim}{{\Psi_{\mathrm{Gproj}(\Lambda)}\mathrm{dim}}}

\newcommand{\Oxdim}{{\Omega_{\X}\mathrm{dim}}}
\newcommand{\Oedim}{{\Omega_{\E}\mathrm{dim}}}
\newcommand{\E}{{\mathcal{E}}}
\newcommand{\F}{{\mathcal{F}}}
\newcommand{\Z}{{\mathbb{Z}}}
\newcommand{\N}{{\mathbb{N}}}
\newcommand{\I}{{\mathcal{I}}}
\newcommand{\Q}{{\mathcal{P}}}
\newcommand{\GP}{{\mathcal{GP}}}
\newcommand{\GI}{{\mathcal{GI}}}

\newcommand{\C}{{\mathcal{C}}}
\newcommand{\A}{{\mathcal{A}}}
\newcommand{\Ab}{{\mathcal{A}b}}
\newcommand{\B}{{\mathcal{B}}}
\newcommand{\X}{{\mathcal{X}}}
\newcommand{\uC}{\underline{\mathcal{C}}}
\newcommand{\Y}{{\mathcal{Y}}}
\newcommand{\add}{{\mathrm{add}}}
\newcommand{\Tr}{{\mathrm{Tr}}}
\newcommand{\RR}{{\mathrm{RightRes}}}
\newcommand{\LR}{{\mathrm{LeftRes}}}
\newcommand{\rad}{{\mathrm{rad}}}

\newcommand{\Fact}{{\mathrm{Fact}}}

\newcommand{\K}{{\mathrm{K}}}

\newenvironment{dem}{\noindent {\bf Proof.}}{\hfill $\Box$\\}

\usepackage{latexsym,amssymb,amscd}
\usepackage{amsmath}
\usepackage[all]{xy}
\begin{document}
\title[Relative Igusa-Todorov functions]{ Relative Igusa-Todorov functions and relative 
homological dimensions}
\thanks{{\it{$2010$ Mathematics Subject Classifications.}} Primary 16E10, 18E10, 18G25. Secondary 16E05, 16E30.\\
The authors thanks  the Project PAPIIT-Universidad Nacional Aut\'onoma de M\'exico IN102914. This work has been partially supported by  project MathAmSud-RepHomol.
}
\author{ Marcelo Lanzilotta,\\ Octavio Mendoza}
\date{}

\maketitle

\centerline{\small{We dedicate this paper to Eduardo Marcos on
 his sixtieth birthday}}
 
\begin{abstract}
We develope the theory of the $\E$-relative Igusa-Todorov functions in an exact $ IT$-context $(\C,\E)$ (see Definition \ref{IT-context}).  In the case when $\C=\modu\, (\Lambda)$ is the category of finitely generated left 
 $\Lambda$-modules, for an artin algebra $\Lambda,$ and $\E$  is the class of all exact sequences in $\C,$ we recover 
 the usual Igusa-Todorov functions \cite{IT}.  We use the setting of the exact structures and the 
 Auslander-Solberg relative homological theory to generalise the original Igusa-Todorov's results. Furthermore, we introduce the 
 $\E$-relative Igusa-Todorov dimension and also we obtain relationships with the relative global and relative finitistic dimensions and the Gorenstein homological dimensions.
\end{abstract}

\setcounter{tocdepth}{1}
\tableofcontents

\section*{Introduction}

Let $\Lambda$ be an artin algebra. Let us denote by $\modu\,(\Lambda)$ the category of finitely generated left $\Lambda$-modules. 
We use $\pd\,M$ to denote the projective dimension of any $M\in\modu\,(A).$ Recall that
 $$\findim\,(\Lambda):=\sup\{\pd\,M\;:\;M\in\modu\,(\Lambda)\text{ and } \pd\,M<\infty\}$$ is the  finitistic dimension of $\Lambda,$ and also that $$\gldim\,(\Lambda):=\sup\{\pd\,M\;:\;M\in\modu\,(\Lambda)\}$$ is the global dimension of $\Lambda.$ The interest in the finitistic dimension is because of the
 ``finitistic dimension conjecture", which is still open, and states that {\it{the finitistic dimension
 of any artin algebra is finite.}} This conjecture is one of the main
problems in the representation theory of algebras and has depth connections with
the solution of several others important conjectures which are still open. The main
point is that all of those conjectures would be valid if so would be the  finitistic
dimension conjecture. Up to now, only several special cases, for this conjecture, are
verified. Furthermore, as much as we know, only several methods exist in order to detect the finiteness of the finitistic dimension of some given algebra. The reader could look in \cite {HZ}, \cite{XiXu}, and references therein, for the development related with the finitistic dimension conjecture.
\

In \cite{IT}, K. Igusa and G. Todorov defined two functions $\Phi,\Psi:\modu\,(\Lambda)\to\mathbb{N},$ known as 
Igusa-Todorov functions ($ IT$-functions, for short).  These 
$ IT$-functions determine new homological measures, generalising the notion of projective dimension, 
and have become a powerful tool to understand better the finitistic dimension conjecture. A lot of new 
ideas have been developed around the use of $ IT$-functions \cite{FLM, HL,HLM,HLM2,HLM3,W,Wei,Xi1,Xi2,Xi3}. 
\

The notion and the use of the $ IT$-functions has been expanding to other settings, for example into 
the bounded derived category \cite{Xu} and finite dimensional comodules for left semi-perfect coalgebras \cite{HaLM}. 
\

Following \cite{HL} and \cite{HLM}, and taking the supremum over $\modu\,(\Lambda)$ of $\Phi$ and $\Psi,$ respectively, we get the Igusa-Todorov dimensions $\Phi\mathrm{dim}\,(\Lambda)$ and $\Psi\mathrm{dim}\,(\Lambda).$ Since the functions $\Phi$ and $\Psi$ are  refinements of the projective dimension (see {\cite{IT}), we have   
$\findim\,(\Lambda)  \leq \Phi\mathrm{dim}\,(\Lambda) \leq\Psi\mathrm{dim}\,(\Lambda)\leq \gldim\,(\Lambda).$
Therefore, by taking into account the previous inequalities, it is settled in \cite{FLM}  the ``$\Psi$-dimension conjecture": 
{\it The $\Psi$-dimension of any artin algebra is finite.}
Observe that, the $\Psi$-dimension conjecture implies the finitistic dimension conjecture; and hence it could be used as a tool to deal with the finitistic dimension conjecture. 
\

The study of  $ IT$-functions is already interesting in itself. Surprisingly, as was discovered by F. Huard and M. Lanzilotta  in \cite{HL},  the Igusa-Todorov functions can be used to characterise self-injective algebras. If we want to go further, a natural question arise: is it possible to 
define $ IT$-functions in a more general contexts as exact categories? In this paper, we show that such  a generalisation is possible into the setting of special exact categories that we call $ IT$-contexts. In such $ IT$-contexts, we develop the theory of the relative $IT$-functions, and correspondingly, the theory 
of the relative Igusa-Todorov dimensions. One of the results we got is that Frobenius categories can be characterised by using relative $ IT$-functions, generalising some of the results obtained in \cite{HaLM, HL} for quasi-co-Frobenius coalgebras and selfinjective algebras, respectively. Moreover, some important homological dimensions (as the 
representation dimension introduced by M. Auslander in \cite{A2}) can be seen as a particular case of a relative Igusa-Todorov dimension. Therefore, the generalization of $ IT$-functions allows us to obtain deeper connections with other relative homological dimensions. 
\

In Section 1,  we collect all the needed background material that is necessary for the development of the paper. We start with approximation theory in the sense of M. Auslander and R. O. Buchweitz \cite{ABu}, and here we recall several relative homological dimensions.  We continue with the notion of cotorsion pairs and we also deal with the theory of exact categories by recalling main definitions and  some fundamental results as: Horseshoe, comparison and Schanuel Lemmas. After that, we construct exact categories from abelian ones. To finish this section, some  concepts of stable categories and  relative syzygy functors needed in this paper are introduced.
\

In Section 2, we define and develop the medullary part of the theory of relative Igusa-Todorov functions. To do that, it is introduced the notions of exact and abelian $IT$-contexts. We state and prove their basic properties, for exact categories, the corresponding generalisations of those functions given by K. Igusa and G. Todorov in \cite{IT}. 
\

In Section 3, we introduce the relative Igusa-Todorov dimensions $\Phi_\E\mathrm{dim}(\C)$ and $\Psi_\E\mathrm{dim}(\C)$  of an 
 exact $IT$-context $(\C,\E).$ Dually, for an exact $IT^{op}$-context $(\C,\E),$ we have the relative Igusa-Todorov dimensions $\Phi^\E\mathrm{dim}(\C)$ and $\Psi^\E\mathrm{dim}(\C).$ These dimensions are new relative homological dimensions encoding fundamental properties of the given category $\C.$ In Theorem \ref{ThmF1}, we characterise small abelian Frobenius KS-categories  in terms of their relative $IT$-dimensions. It is also presented several connections between these relative $IT$-dimensions and the usual relative homological dimensions. We also deal with the representation dimension of an algebra. In this section, we show that the representation dimension  can be seen as a particular case of a relative $IT$-dimension. We also get as a corollary, see Corollary \ref{rkbasico1.1},  the main result of \cite{IT}.
\

Section 4 is devoted to the study of the relative syzygy functor $\Omega_\E:{}^\perp\Q(\E)\to {}^\perp\Q(\E),$ where $(\C,\E)$ is an 
exact $IT$-context which is idempotent complete, and  $\Q(\E)$ is the class of the $\E$-projectives. Inspired by \cite{AB},  we prove that such a functor is full and faithful. As  a first consequence, we get that 
$\Phi_\E\mathrm{dim}(^\perp\Q(\E))=\Psi_\E\mathrm{dim}(^\perp\Q(\E))=0.$ Using the above equality, we were able to prove one of the main result of this section:  $\fpd_\E(\C)\leq\Fiedim\,(\C)\leq\Psiedim\,(\C)\leq \id_\E(\Q(\E)).$ Finally, we apply to different contexts the preceding inequalities obtaining results for  standardly stratifed algebras (see Section 7) and for any artin algebra with finite injective dimension.
\

The material presented in Section 5 deals with  $\Ext^i_\X(-,-)$ and $\Ext^i_\Y(-,-)$ which are the right derived functors 
 of $\Hom(-,-),$ for some ``good" pair of classes $\X$ and $\Y.$ We discuss the balance of these functors on different contexts, namely:
 left (right) compatible pairs, homologically compatible pairs and homologically compatible cotorsion pairs.  In this section, we obtain general theorems relating  relative $ IT$-dimensions and classical relative homological dimensions. 
\

We dedicate Section 6  to define the relative $n$-Igusa-Todorov categories, generalising the ideas in \cite{Wei}. It is also proven that the relative finitistic dimension of any $n$-IT category is finite. 
\

In Section 7, we give applications of the developed theory.  We specify some concrete types of $ IT$-contexts coming from the 
abelian category $\modu\,(\Lambda),$ where $\Lambda$ is an artin algebra. We start with tilting and cotilting $\Lambda$-modules, 
and after that, we consider the case when $\Lambda$ is a standardly stratified algebra. Finally, we also connect the relative $ IT$-dimensions 
with the Gorenstein projective (injective) dimension. In case $\Lambda$ is a Gorenstein algebra, we get Theorem \ref{GorAl}  where, 
in particular, appears the equalities 
\

$\Phi_{\Gproj(\Lambda)}\mathrm{dim}(\Lambda)=\PsiGpdim(\Lambda)=\Gid(\Gproj(\Lambda))=\Gpd(\Ginj(\Lambda))=\\=
\Phi^{\Ginj(\Lambda)}\mathrm{dim}(\Lambda)=\Psi^{\Ginj(\Lambda)}\mathrm{dim}(\Lambda)=\id\,({}_\Lambda\Lambda).$
\

The last result in this section is Theorem \ref{T1Sbalance2}. Here, we consider the case of a resolving, coresolving and functorially 
finite class $\mathcal{U}$ in $\modu\,(\Lambda).$ 
 \

\section{Relative homology}

We start this section by collecting all the background material that will be necessary in the sequel. 
\

 Let $\C$ be an additive category. For any  decomposition $C=X\oplus Y$ of $C\in\C,$ 
 we say that $X$ and $Y$ are direct summands of $C,$ and to say that $X$ is a direct summand of $C,$ we just write $X\mid C.$ 
\

Assume that $\C$ is an abelian category. We say that a class $\X$ of objects in $\C$ is {\bf closed under extensions}, if for any exact sequence $0\to A\to B\to C\to 0,$ with $A,C\in\X,$ we have that $B\in\X.$ Note that a closed extension class $\X,$ containing the zero element, is closed under isomorphisms and finite coproducts. 
\bigskip

{\sc Approximations.}
Let  $\X\subseteq \C$ be a class of objects in an abelian category $\C.$ A morphism $f:X\to C$ is said to be an {\bf $\X$-precover} of $C$ if 
$X\in\X$ and the map $\Hom_\C(X',f):\Hom_\C(X',X)\to\Hom_\C(X',C)$ is surjective for any $X'\in\X.$ In case any object $C\in\C$ 
admits an $\X$-precover $f:X\to C,$ it is said that $\X$ is a {\bf precovering class} in $\C.$ An $\X$-precover $f:X\to C$ is said to be 
an {\bf $\X$-cover} if the morphism $f$ is {\bf right-minimal}, that is,  for any $\alpha\in\End_\C(X)$ such that $f\alpha=f,$ it follows that $\alpha$ 
is an isomorphism. Note that, in case there are two $\X$-covers $f:X\to C$ and $f':X'\to C$ of $C,$ then there exists an 
isomorphism $\alpha:X\to X'$ such that $f'\alpha=f.$ Hence, in case there exists, an $\X$-cover of $C$ is unique up to isomorphism.\\ 
We will  freely use  the dual notion of an {\bf $\X$-preenvelope} (respectively, {\bf $\X$-envelope}) of an object $C\in\C,$ 
and also the notion of  {\bf preenveloping class} in $\C.$ Finally, we say that a class $\X\subseteq\C$  is {\bf functorially finite} if it is precovering and preenveloping.

{\sc Resolution and coresolution dimension.}
Let $C$ be an object in an abelian category $\C$ and  $\X\subseteq\C$ be a class of objects in $\C.$ Following Auslander-Buchweitz \cite{ABu}, the 
$\X$-{\bf{resolution dimension}} $\resdim_\X\,(C)$ of $C$ is the minimal non-negative integer $n$ such that there is an 
exact sequence $$0\to X_n\to\cdots\to X_1\to X_0\to C\to 0$$ with $X_i\in\X$ for $0\leq i\leq n.$ If such $n$ does not exist, we 
set $\resdim_\X\,(C):=\infty.$ We shall denote by $\X^\wedge$ the class of all objects $C\in\C$ 
such that $\resdim_\X\,(C)<\infty.$ Dually, we have the $\X$-{\bf{coresolution dimension}} $\coresdim_\X\,(C)$ of $C$ and the class $\X^\vee$ of all objects having finite $\X$-coresolution dimension. 
\

The {\bf{total $\X$-resolution dimension}} $\tresdim_\X\,(C)$ of $C$ is the minimal non-negative integer $n$ such that there is an 
exact sequence $$\eta_C:\quad 0\to X_n\to\cdots\to X_1\to X_0\to C\to 0$$ with $X_i\in\X,$ for $0\leq i\leq n,$ such that the complex $\Hom_\C(X,\eta_C)$ is acyclic for any $X\in\X.$ If such $n$ does not exist, we set $\tresdim_\X\,(C):=\infty.$ Note that, in general, we only have that 
$\resdim_\X(M)\leq\tresdim_\X(M),$ since any total $\X$-resolution is an $\X$-resolution.
\

 Following \cite{ABu}, we recall that a class $\omega\subseteq\C$ is 
 {\bf $\X$-injective}  if 
  $$\X\subseteq{}^\perp\omega:=\{M\in\C\;:\; \Ext_\C^i(M,W)=0\;\forall\,i\geq 1,\,
  \forall\,W\in\omega\}.$$ 
 We say that $\omega$ is a {\bf relative cogenerator} in $\X$ if  $\omega\subseteq\X$ and for any $X\in\X$ there is 
 an exact sequence $0\to X\to W\to X'\to 0$ with $W\in\omega$ and $X'\in\X.$ Dually, we have the notion of {\bf $\X$-projective} and of {\bf relative generator} in $\X.$ 

\begin{lem}\label{tresdim} Let $\X$ be a class of objects in $\C,$ which is closed under extensions and 
$0\in\X,$ and let $\omega$ be an $\X$-injective relative cogenerator in $\X.$ Then, the following statements hold true.
\begin{itemize}
\item[(a)]  For any $M\in\X^{\wedge}$ there is an exact sequence $0\to K_0\to X_0\stackrel{f_0}{\to}M\to 0$ in $\C,$ where $f_0:X_0\to M$ is an $\X$-precover of $M$ and $\resdim_\omega(K_0)=\resdim_\X(M)-1.$
\item[(b)] $\resdim_\X(M)=\tresdim_\X(M)\quad\;\text{for any}\;M\in\C.$
\end{itemize}

\end{lem}
\begin{dem} (a) The proof given in \cite[Theorem 1.1]{ABu} can be adapted to get (a).
\

(b) We already have that $\resdim_\X(M)\leq\tresdim_\X(M).$ Assume that $n:=\resdim_\X(M)<\infty.$ Then, by (a) there is an exact sequence $0\to K_0\to X_0\stackrel{f_0}{\to}M\to 0$ in $\C,$ where $f_0:X_0\to M$ is an $\X$-precover of $M$ and $\resdim_\omega(K_0)=n-1.$  In particular, using that $\omega\subseteq \X,$ we have that $K_0\in\X^\wedge.$   Applying the item (a) several times, we get a total $\X$-resolution 
$0\to X_n\to\cdots\to X_1\to X_0\to M\to 0,$ proving that $\tresdim_\X(M)\leq n.$  
\end{dem}

{\sc Relative projective dimensions in abelian categories.} Let $\C$ be an abelian category and $\X$ be a class of objects in $\C.$ Following M. Auslander and R. O. Buchweitz \cite{ABu}, for any object $M\in\C,$ the {\bf{relative projective dimension}} of $M$ with respect to $\X$ is defined as $$\pd_{\X}\,(M):=\mathrm{min}\,\{n\geq 0\,:\,\Ext_\C^j(M,-)|_{\X}=0  \text{ for any } j>n\}.$$ Dually, we denote by
  $\mathrm{id}_{\X}\,(M)$ the  {\bf{relative injective dimension}} of
  $M$ with respect to $\X.$ Furthermore, for any class $\Y\subseteq\C,$ we set 
  $$\pd_\X\,(\Y):=\mathrm{sup}\,\{\pd_\X\,Y\;:\; Y\in\Y\}\text{ and }\id_\X\,(\Y):=\mathrm{sup}\,\{\id_\X\,Y\;:\; Y\in\Y\}.$$ 
  If $\X=\C,$ we just write $\pd\,(\Y)$ and   $\id\,(\Y)$.
 \bigskip
 
   The following basic property is straightforward.

\begin{lem}\label{Lcam}   \cite{ABu} Let $\X$ and $\Y$  be  classes in $\C.$ Then $$\pd_\X\,(\Y)=\id_\Y\,(\X).$$ 
\end{lem} 

{\sc Cotorsion pairs.} Let $\C$ be an abelian category. For any class $\X$ of objects in $\C,$ we consider the classes
${}^{\perp_i}\X:=\cap_{X\in\X}\Ker\,\Ext^i_\C(-,X)$ and $\X^{\perp_i}:=\cap_{X\in\X}\Ker\,\Ext^i_\C(X,-),$ for any non negative 
integer $i.$ We recall that the class $\X$ is {\bf resolving} if it is closed under extensions, kernels of epimorphisms and contains the projectives. Dually, we have the notion of {\bf coresolving} for the class $\X.$ 
\

A pair $(\A,\B)$ of classes of objects in $\C$ is  a {\bf cotorsion pair} if 
$\A={}^{\perp_1}\B$  and $\B=\A^{\perp_1}.$ 
\

A cotorsion pair $(\A,\B)$ is {\bf complete} if  for any $C\in\C$ there are exact sequences 
$0\to B\to A\to C\to 0$ and $0\to C\to B'\to A'\to 0,$ for some $A,A'\in\A$ and $B,B'\in\B.$ Recall that the 
cotorsion pair $(\A,\B)$  is {\bf hereditary} if $\A$ is resolving and $\B$ is coresolving. In such a case, 
by the shifting argument, it can be seen that ${}^{\perp_1}\B={}^{\perp}\B:=\cap_{i\geq 1}{}^{\perp_i}\B$ and 
$\A^{\perp_1}=\A^{\perp}.$ In particular, $\Ext^i_\C(A,B)=0$ for any $i\geq 1,$ $A\in\A$ and $B\in\B.$

{\sc Exact categories.} Exact categories has become nowadays an important tool. In this paragraph, 
we introduce the necessary development for the paper. The reader can see, for example, in 
\cite{Buhler} and \cite{DRSS} and the  therein citations. 
\

We always assume that $\C$ is an additive category. A kernel-cokernel pair $(i,p)$ in $\C$ is a pair of composable morphisms
$A'\stackrel{i}{\to}A\stackrel{p}{\to}A''$ such that $i=\Ker(p)$ and $p=\Coker(i).$ An exact structure on $\C$ is a class $\mathcal{E}$ of kernel-cokernel pairs satisfying a list of six axioms 
(see \cite[Definition 2.1]{Buhler}). An {\bf exact category} is a pair $(\C,\mathcal{E}),$ where $\C$ is an additive 
category and $\mathcal{E}$ is an exact structure on $\C.$ An element $(i,p)\in\mathcal{E}$ is 
 usually called short $\E$-exact sequence and written as $A'\stackrel{i}{\rightarrowtail}A\stackrel{p}{\twoheadrightarrow}A'',$ in such a case $i$ is {\bf admissible monic} 
 and $p$ is {\bf admissible epic}. Sometimes, if it is clear from the context, we also write 
 $A'\stackrel{i}{\to}A\stackrel{p}{\to}A'',$ for the pair $(i,p)\in\E.$

 The philosophy of an exact structure, on a given additive category, is to define what will be the replacement for the short exact sequences, which are usually defined in abelian categories. Let $\C$ be an abelian category. Then, we have al least two 
 exact structures on $\C,$ namely: the class $\mathcal{\E}_{\min}$ of all split-exact sequences and the class $\mathcal{\E}_{\max}$ of all exact sequences. Furthermore, for any full additive subcategory 
 $\X\subseteq\C,$ closed under extensions, the class $\mathcal{E}_\X$ of all exact sequences in 
 $\C$ with terms in $\X$ is 
 also an exact structure on $\C.$ It can be shown (see \cite[Appendix A]{Keller}) that each skeletally 
 small exact category $(\C,\E)$ admits an equivalence $F:\C\to \X,$ with $\X$ a full extension closed subcategory of 
 an abelian category $\A,$ such that for any pair $(i,p)$ of composable morphisms in $\C,$ we have that 
 $(i,p)\in\E$ if and only if $0\to F(A')\xrightarrow{F(i)}F(A)\xrightarrow{F(p)}F(A'')\to 0$ is an exact sequence in $\A.$  
\

Let $r:B\to C$ and $s:C\to B$ be morphisms in $\C$ such that $rs=1_C.$ In this case, we say that $r$  is split-epi and $s$ is split-mono. The additive category $\C$ is {\bf weakly idempotent complete}  if any split-mono in $\C$ has a cokernel. Note that (see \cite[7.1]{Buhler}), $\C$ is weakly idempotent complete if an only if any split-epi in $\C$ has a kernel. This notion will be used in Theorem \ref{experp1.5}.

\begin{remark}\label{wic} \cite[Corollary 7.4, Proposition 7.5]{Buhler} For any exact category $(\C,\mathcal{E}),$ the following statements are equivalent.
\begin{itemize}
\item[(a)] $\C$ is weakly idempotent complete.
\item[(b)] Any split-mono is admissible monic.
\item[(c)] Any split-epi is admissible epic.
\item[(d)] Let $(f,g)$ be any pair of composable morphisms. If $gf$ is admissible epic then 
$g$ is admissible epic.
\end{itemize}
\end{remark}

An object $P$ in an exact category $(\C,\mathcal{E})$ is {\bf $\mathcal{E}$-projective} if the Hom functor $\Hom_\C(P,-):\C\to \Ab$ takes 
any short $\E$-exact sequence to a short exact sequence of abelian groups. The class of all 
$\mathcal{E}$-projective objects is denoted by $\Q(\E).$ It is said that $\C$ has {\bf enough $\mathcal{E}$-projectives} if for any $C\in\C$ there is an admissible epic $P\twoheadrightarrow C$ 
with $P\in\Q(\E).$ Note that, for $P=\oplus_{i\in I}P_i$ in $\C,$ we have that  $P\in\Q(\E)$ if and only if $P_i\in\Q(\E)$ for 
any $i\in I.$ In particular $\add\,\Q(\E)=\Q(\E).$ Dually, we have  the
{\bf $\mathcal{E}$-injectives}, the class $\I(\E)$ of the $\mathcal{E}$-injective objects and the 
the notion of {\bf enough $\mathcal{E}$-injectives} for $\C.$
\

The following statement is the Horseshoe's Lemma for exact categories.
\begin{lem}\label{horseshoe}
Let
$\xymatrix@R0.4cm@C=0.6cm{
& A' \ar[d]^{\beta_0} & & C'\ar[d]^{\beta_1}\\
 & P_0 \ar[d]^{\alpha_0}&  & P_1\ar[d]^{\alpha_1} \\
& A\ar[r]_f  & B\ar[r]_g & C 
 }$

\noindent be an $\E$-exact diagram in an exact category $(\C,\E)$ (i.e. all rows and columns are short $\E$-exact sequences) with $P_0, P_1\in \Q(\E)$. Then, there is an $\E$-exact and commutative diagram in $\C$
$$\xymatrix@R0.8cm@C=0.6cm{  
 A' \ar[r]  \ar[d]_{\beta_0} & B'\ar[r]\ar[d] & C'\ar[d]^{\beta_1} \\
 P_0 \ar[d]_{\alpha_0} \ar[r] & P_0\oplus P_1\ar[r]^{\qquad(0,1)\quad} \ar[d]^{\lambda} & P_1\ar[d]^{\alpha_1}\\
A\ar[r]_f  & B\ar[r]_g & C  }$$
\end{lem}
\begin{dem} See \cite[Theorem 12.9]{Buhler}.
\end{dem}

A chain complex $M_\bullet:\quad \cdots\rightarrow M_{n+1}\stackrel{d_{n+1}}{\rightarrow} M_{n}\stackrel{d_{n}}{\rightarrow} M_{n-1}\rightarrow  \cdots $ in $\C$, is said to be 
{\bf $\E$-acyclic} 
if each differential $d_n$ factors as $M_n\twoheadrightarrow Z_{n-1}(M_\bullet)\rightarrowtail M_{n-1}$ in such a way that each sequence $Z_n(M_\bullet)\rightarrowtail M_n\twoheadrightarrow Z_{n-1}(M_\bullet)$ is $\E$-exact.
\begin{defi} \cite{Buhler} An {\bf $\E$-projective resolution} of an object $C\in\C$ is an $\E$-acyclic complex
$$ \cdots \rightarrow P_n \rightarrow P_{n-1} \rightarrow \cdots \rightarrow P_1 \rightarrow P_0\stackrel{\varphi}{\twoheadrightarrow} C,$$
where $P_i\in\Q(\E)$ for any $i\in \N$. We 
 write this resolution as $P_{\bullet}(C)\stackrel{\varphi}{\twoheadrightarrow} C.$
\end{defi}
\begin{remark}\label{enoughP} If the exact category $(\C,\E)$ has enough $\E$-projectives, then 
any $C\in\C$ has an $\E$-projective resolution $P_{\bullet}(C)\stackrel{\varphi}{\twoheadrightarrow} C.$
\end{remark}

The following result is known as the comparison Lemma.

\begin{lem}\label{comparison} Let $P_{\bullet}(M)\stackrel{\varphi}{\twoheadrightarrow} M$ and 
$Q_{\bullet}(N)\stackrel{\eta}{\twoheadrightarrow} N$ be $\E$-projective resolutions. Then, 
for any morphism $f:M\to N$ in $\C$ there is a chain complex morphism $\overline{f}:P_{\bullet}(M)\to Q_{\bullet}(N)$ such that $f\varphi=\eta\overline{f}_0.$ Moreover this chain complex 
morphism is unique up to homotopy. 
\end{lem}
\begin{dem} See \cite[Theorem 12.5]{Buhler}
\end{dem}
 
 Assume that the exact category $(\C,\E)$ has enough $\E$-projectives. Let $C$ be an object in $\C.$ Then, by  Lemma \ref{comparison} it can be proven that, for any object $B$ in the category $\C,$  the homology group 
 $H^n(\Hom(P_{\bullet}(C),B))$ does not depend on a chosen $\E$-projective resolution 
 $P_{\bullet}(C)\twoheadrightarrow C.$ Thus, 
 we can define the $\E$-relative $n$-extension group 
 $$\Ext_\E^{n}(C,B):=H^n(\Hom(P_{\bullet}(C),B)).$$
  Note that the Horseshoe's Lemma (see Lemma \ref{horseshoe}) implies the existence of 
 the long exact sequence in homology induced by the functor $\Hom(-,B).$ That is \cite[Remark 12.12]{Buhler}
  any $\E$-exact sequence $A'\rightarrowtail A\twoheadrightarrow A''$ gives a long exact sequence 
 $$\cdots\to\Ext^{i}_\E(A'',B)\to\Ext^{i}_\E(A,B)\to\Ext^{i}_\E(A',B)\to\Ext^{i+1}_\E(A'',B)\to\cdots,$$
 where $\Ext^{0}_\E(X,Y)=\Hom_\C(X,Y)$ for any $X,Y\in\C.$
 \
  
\begin{defi} For a given exact category $(\C,\E)$ with enough $\E$-projectives, we introduce the following dimensions for any $C\in\C.$
\begin{itemize}
\item[(a)] The {\bf $\E$-syzygy dimension} $\Omega_\E\mathrm{dim}(C)$ is the minimal non-negative integer $n$ such that there is an $\E$-projective resolution of $C$
 $$P_n \rightarrowtail P_{n-1} \rightarrow \cdots \rightarrow P_1 \rightarrow P_0 \twoheadrightarrow C.$$
\item[(b)] The {\bf $\E$-projective dimension}
$$\pd_\E(C):=\min\{n\in \N\ :\ \Ext_\E^{j}(C,Z)=0,\ \forall j>n,\ \forall Z\in \C\}.$$
\item[(c)] The {\bf $\E$-injective dimension}
$$\id_\E(C):=\min\{n\in \N\ :\ \Ext_\E^{j}(Z,C)=0,\ \forall j>n,\ \forall Z\in \C\}.$$
\end{itemize}
\end{defi}

\begin{lem}\label{primeraI} Let $(\C,\E)$ be an exact category with enough $\E$-projectives. Then, 
for any $C\in\C,$ we have that 
$$\Omega_\E\mathrm{dim}(C)=\pd_\E(C).$$
\end{lem}
\begin{dem} We do as in the usual case. Using the 
long exact sequence given by the right derived functors $\Ext_\E^{j}(-,Z),$ for any $Z\in\C,$ and the shifting argument for $\E$-extensions groups, we can prove that $\Omega_\E\mathrm{dim}(C)\leq n$ if and only if $\pd_\E(C)\leq n.$
\end{dem}

For each class $\mathcal{Y}$ of objects in $\C$, we set 
$$\pd_\E(\Y):=\mathrm{sup}\{\pd_\E(Y)\ :\ Y\in \Y\}.$$
 Following \cite{AS2}, we introduce the  definition of the finitistic $\E$-projective dimension of a given class of objects in $\C.$

\begin{defi}
For each class $\mathcal{Y}$ of objects in $\C$, we set 
$$\mathcal{P}^{<\infty}_{\E}(\Y):=\{Y\in\Y\;:\;\pd_\E(Y)<\infty\}.$$
The {\bf finitistic $\E$-projective} dimension of $\Y$ is the following
$$\mathrm{fpd}_\E(\Y):= \pd_\E(\mathcal{P}^{<\infty}_{\E}(\Y)).$$
Similarly, we have the class $\I^{<\infty}_{\E}(\Y)$ and the {\bf finitistic $\E$-injective} dimension $\mathrm{fid}_\E(\Y).$
\end{defi}

Assume that $\C$ is abelian and $\E=\E_{\max}$ is the exact structure given by all the exact sequences in $\C.$  In this case, $\E_{\max}$-projective dimension is just the ordinary projective dimension, and we just write $\fpd\,(\Y)$  and call this number the finitistic projective dimension of the class $\Y.$  Similarly, we have the finitistic injective dimension $\fid\,(\Y)$ of $\Y.$
\vspace{0.2cm}

{\sc  Exact categories from abelian categories.} 

Let $\C$ be an abelian category. The class of all {\bf projective} objects in $\C$ is 
usually denoted by $\Q(\C).$ 

A {\bf generator} in an abelian category $\C$ is a class $\X$ of objects in $\C$ such that $\add\,(\X)=\X$ and 
$\Q(\C)\subseteq \X.$
\

In this subsection, we assume that $\C$ is an abelian category $\C.$ Some of the results given in \cite{AS1}, by M. Auslander and {\O}. Solberg, can be adapted (as can be seen in \cite{DRSS}) into our setting of the  abelian category $\C$  and a generator  precovering class $\X$ of objects in $\C.$ 
\

A class $\X$ of objects in $\C$ defines \cite[Proposition 1.7]{AS1} an {\bf additive subfunctor} $F:=F_{\X}$ of $\Ext^{1}_{\C}(-,-)$ as follows: For any pair 
$(C,A)$ of objects in $\C$, the abelian group $F(C,A)$ consists of all short exact sequences
$0\rightarrow A\rightarrow B\rightarrow C\rightarrow 0$ in $\C$ such that the induced map 
$\Hom_{\C}(-,B)|_\X \rightarrow \Hom_{\C}(-,C)|_\X $ is surjective. Such an exact sequence is called an {\bf $F$-exact sequence}. We denote by $\E_F$ the class of all $F$-exact sequences in $\C.$

\begin{pro}\label{projExcat} Let $\C$ be an abelian category with enough projectives, $\X$ be a 
generator precovering class in $\C,$ and let $F:=F_{\X}$ be the associated subfunctor 
of $\Ext^{1}_{\C}(-,-).$  Then, the pair $(\C,\E_F)$ is an exact category with enough $\E_F$-projectives and such that $\Q(\E_F)=\X.$
\end{pro}
\begin{dem} By \cite[Corollary 1.6, Proposition 1.7]{DRSS}, it follows that the pair $(\C,\E_F)$ is an exact category. The 
conditions over the class $\X$ implies that, for any $C\in\C,$ there is an exact sequence 
$\eta_C:\quad 0\to K_C\to X_C\stackrel{\mu_C}\to C\to 0$ such that $\mu_C$ is 
an $\X$-precover of $C.$ In particular $\eta_C\in\E_F,$ and since $\X\subseteq\Q(\E_F),$ 
we conclude that the exact category has enough $\E_F$-projectives.
\

 Let $Q\in\Q(\E_F).$ Consider the $F$-exact sequence as above 
 $$\eta_Q:\quad 0\to K_Q\to X_Q\stackrel{\mu_Q}\to Q\to 0.$$ 
Since $Q$ is $\E_F$-projective, the sequence $\eta_Q$ splits and thus $Q\mid X_Q$. Therefore  $Q\in\X,$ because we know that $\add(\X)=\X.$
\end{dem}

For the exact category $(\C,\E_F),$ having enough $\E_F$-projectives and $F:=F_\X,$ we have: 
the $F$-relative $n$-extension group $\Ext^n_F(A,B):=\Ext^n_{\E_F}(A,B)$ (some times also written 
as $\Ext^n_\X(A,B)$); the $F$-projective dimension $\pd_F(M):=\pd_{\E_F}(M);$ the $F$-injective dimension $\id_F(M):=\id_{\E_F}(M);$ and finally the $\X$-syzygy dimension 
$\Omega_\X\mathrm{dim}(M):=\Omega_{\E_F}\mathrm{dim}(M).$ Moreover, we have 
$\mathcal{P}^{<\infty}_{F}(\Y):=\mathcal{P}^{<\infty}_{\E_F}(\Y)$ and the 
finitistic $F$-projective dimension $\mathrm{fpd}_F(\Y):=\fpd_{\E_F}(\Y).$ Similarly, we have the finitistic $F$-injective dimension 
$\mathrm{fid}_F(\Y):=\finid_{\E_F}(\Y).$
\

Assume that $\C:=\modu\,(\Lambda)$ is the category of finitely generated $\Lambda$-modules, for an artin algebra $\Lambda.$  In this case, the $F$-finitistic dimension $\findim_F(\Lambda)$ of the algebra $\Lambda$ is just $\fpd_F(\modu\,(\Lambda)).$ Finally, the $F$-global dimension 
$\gldim_F(\Lambda)$ of the algebra $\Lambda$ is just $\pd_F(\modu\,(\Lambda)).$ In case $\X=\proj\,(\Lambda):=\Q(\modu\,(\Lambda)),$ we get, 
 respectively, the finitistic dimension $\findim\,(\Lambda)$ and the global dimension $\gldim\,(\Lambda)$ of $\Lambda.$  

The following result will be very useful in the sequel.

\begin{pro}\label{igualdades} Let $\C$ be an abelian category with enough projectives, and $\X$ 
be a generator  precovering class in $\C.$ Then, for $F:=F_\X,$ the following statements 
hold true.
\begin{itemize}
\item[(a)] For any $M\in \C$, we have that 
$$\resdim_\X(M)\leq\Oxdim(M)=\pdF(M).$$
\item[(b)] Let $\X$ be closed under extensions and let  $\omega$ be an $\X$-injective relative cogenerator 
in $\X.$ Then
 \begin{itemize}
 \item[(b1)] $\resdim_\X(M)=\pdF(M)\quad\text{for any}\quad M\in\C,$
 \item[(b2)] $\fpd_F(\C)=\pd_F(\X^\wedge)=\id_{\X^\wedge}\,(\omega)\leq\id\,(\omega).$
 \end{itemize}
\end{itemize}
\end{pro}
\begin{dem} (a) By Proposition \ref{projExcat}, we know that $\Q(\E_F)=\X.$ Therefore any 
$\E_F$-projective resolution $P_{\bullet}(C)\twoheadrightarrow C$ is in particular an exact 
$\X$-resolution. Thus, we get $\resdim_\X(M)\leq \Oxdim(M).$ The equality $\Oxdim(M)=\pdF(M)$ follows from Lemma \ref{primeraI}. 
\

(b1) Note, firstly, that $\Oxdim(M)=\tresdim_\X(M).$ Since $\omega$ is an $\X$-injective relative cogenerator in $\X,$  from Lemma \ref{tresdim}, we get that  
that $\tresdim_\X(M)=\resdim_\X(M).$ Thus, (b1) follows from (a).  
\

(b2)  It follows from (b1),  \cite[Proposition 2.1]{ABu} and Lemma \ref{Lcam}.
\end{dem}

In a similar way as we did before, we assume now that the abelian category $\C$ has enough injectives. Denote by $\I\,(\C)$ the class of injective objects in $\C$. A {\bf cogenerator} in $\C$ is  a  class $\Y$ of objects in $\C$  such that 
$\add\,(\Y)=\Y$ and $\I\,(\C)\subseteq \Y.$ 
\

Fix now a preenveloping  a subclass $\Y$ of objects in $\C$. The class $\Y$ defines \cite{AS1} an {\bf additive subfunctor} $G:=F^{\Y}$ of $\Ext^{1}_{\C}(-,-)$ as follows: For any pair 
$(C,A)$ of objects in $\C$, the abelian group $G(C,A)$ consists of all short exact sequences
$0\rightarrow A\rightarrow B\rightarrow C\rightarrow 0$ in $\C$ such that the induced map 
$\Hom_{\C}(B,-)|_\Y \rightarrow \Hom_{\C}(A,-)|_\Y$ is surjective. Such an exact sequence is called a {\bf $G$-exact sequence} and $\E_G$ denotes the class of all $G$-exact sequences. 
Note that the pair $(\C,\E_G)$ is an exact category.  Furtheremore, if $\Y$ is  a preenveloping cogenerator, then $\C$ has enough $\E_G$-injectives and $I(\E_G)=\Y.$ 
In this case,  we have the right derived functor $\Ext^{i}_G(A,-)$ of 
$\Hom_{\C}(A,-),$ and such a functor is also denoted as $\Ext^{i}_\Y(A,-):=\Ext^{i}_G(A,-).$  
\vspace{0.2cm} 

{\sc The stable category modulo projectives.} 
Let $(\C,\E)$ be an exact category with enough $\E$-projectives. In this case, for any 
$C\in\C,$ there is an $\E$-exact sequence $K\rightarrowtail P\twoheadrightarrow C,$ with $P\in\Q(\E).$ The object $K$ 
is said to be an {\bf $\E$-syzygy}  of $C,$ and written also as $\Omega_{\E,P}(C)$ to encode the information in the above $\E$-exact sequence. Note that, in general, $\E$-syzygy objects are not 
uniquely determined, we could have another  $\E$-exact sequence $K'\rightarrowtail P'\twoheadrightarrow C,$ with $P'\in\Q(\E).$ So, there is no reason to get an isomorphism in $\C$ from 
$\Omega_{\E,P}(C)$ to $\Omega_{\E,P'}(C).$
\

The following is the Schanuel's Lemma for exact categories with enough projectives.

\begin{lem}\label{schanuelL} Let $C\in\C.$ Then, for any two $\E$-syzygies of $C,$ we have
$$\Omega_{\E,P}(C)\oplus P'=\Omega_{\E,P'}(C)\oplus P.$$
\end{lem}
\begin{dem} By taking the $\E$-exact sequences, defining these $\E$-syzygies, and using that $P\in\Q(\E),$ we get the following $\E$-exact and commutative diagram
$$\xymatrix@1{  \Omega_{\E,P}(C) \ar[r]\ar[d]  & P \ar[r]\ar[d] & C \ar@{=}[d]\\
                \Omega_{\E,P'}(C)\ar[r]  & P' \ar[r] & C. }$$
Hence, by \cite[Proposition 2.12]{Buhler}, we get the $\E$-exact sequence 
$$\Omega_{\E,P}(C)\rightarrowtail\Omega_{\E,P'}(C)\oplus P\twoheadrightarrow P',$$
which splits since $P'\in\Q(\E).$
\end{dem}

For any $M,N\in \C$, we denote by $\Fact_{\Q(\E)}(M,N)$ the set of all morphisms 
$f: M \rightarrow N$ factoring throughout an object in 
$\Q(\E)$. Let $N\in\C$ and take any admissible epic $\mu:P\twoheadrightarrow N$ with $P\in\Q(\E).$ Note that, for any $M\in \C$ we have that 
$$\Fact_{\Q(\E)}(M,N)=\Ima\,(\Hom_{\C}(M,\mu)).$$
Therefore the class of morphisms $\Fact_{\Q(\E)}$ is an ideal in $\C$, and so, we get the well 
known {\bf stable category} $\underline{\C}:=\C/\Fact_{\Q(\E)}$ {\bf modulo projectives}. Recall that $\underline{\C}$ and $\C$ have the same objects and 
$\uHom(M,N):=\Hom_{\C}(M,N)/\mathrm{Fact}_{\Q(\E)}(M,N).$
\

In what follows, we present some basic properties of the stable category $\uC.$ They are well-known, and for the convenience of the
reader, we include some proofs.
\begin{lem}\label{Primerlema}
For any $M\in \C$, the following statements hold true.
\begin{itemize}
\item[(a)]
$\uEnd(M)=0$ if and only if $M\in \Q(\E).$
\item[(b)] Let $\End_\C(M)$ be a local ring. Then,   $M\not\in \Q(\E)$ if  and only  if $\Fact_{\Q(\E)}(M,M)$ is contained in the 
radical  of $\End_{\C}(M).$
\end{itemize}
\end{lem}
\begin{dem} Let $\mu:P\twoheadrightarrow M$ be an admissible epic with $P\in\Q(\E).$
\begin{itemize}
\item[(a)] Assume that $\uEnd(M)=0$. Then 
$$\End_{\C}(M)=\mathrm{Fact}_{\Q(\E)}(M,M) =\Ima(\Hom_{\C}(M,\mu))$$ and hence $1_M$ factors throughout $\mu$. Therefore $\mu:P\twoheadrightarrow M$ splits, proving that $M\in\Q(\E).$
\

\item[(b)] Suppose that $M\not\in \Q(\E).$ Let $\alpha\in$ Fact$_{\Q(\E)}(M,M)=\Ima(\Hom_{\C}(M,\mu))$. That is, the following 
diagram is commutative: $$\xymatrix@1{  & M \ar[d]^{\alpha} \ar@{.>}[dl]_{\alpha '}\\
P \ar[r]^{{\mu}} & M. }$$

\noindent Now, if $\alpha\not\in\rad(\End_{\C}(M))$, since $\End_{\C}(M)$ is local, we get that $\alpha:M\rightarrow M$ is an isomorphism. But $\alpha = \mu\alpha '$, implies that $\mu$ splits and then $M\in \Q(\E),$ a contradiction. Hence Fact$_{\Q(\E)}(M,M)\subseteq\rad(\End_{\C}(M))$.
\end{itemize}
\end{dem}

Let $\A$ be an additive category. We denote by $\ind\,(\A)$ the full subcategory of  $\A,$ 
whose objects are determined by choosing  one object for each iso-class of  indecomposable objects in $\A.$ We also say
that $\A$ is a Krull-Schmidt category (from now {\bf KS-category}) if 
each object in $\mathcal{A}$  decomposes into a
 finite direct sum  of objects having local endomorphism rings. In particular,  the endomorphism ring of every object is semi-perfect (see, for example, in \cite{CYZ, Krause}).

\begin{pro}\label{indenelcociente} Let $(\C,\E)$ be any exact category with enough $\E$-projectives. Then, the following statements hold true.
\begin{itemize}
\item[(a)] $\Q(\E)$ coincides with the class of all zero objects in $\uC.$
\item[(b)] For any $C\in\C,$ we have canonical isomorphism 
$$\Omega_{\E,P}(C)\simeq\Omega_{\E,P'}(C)$$ in the stable category $\uC.$
\item[(c)] If $\C$ is Krull-Schmidt, then the stable category $\uC$ is also Krull-Schmidt. Moreover $\ind\,(\uC)=\ind\,(\C)-\Q(\E)$.
\end{itemize}
\end{pro}
\begin{dem} (a) It follows from Lemma \ref{Primerlema} (a).
\

(b) Since $\Q(\E)$ coincides with the class of all zero objects in $\uC,$ then Schanuel's Lemma (see 
Lemma \ref{schanuelL}) gives us (b).
\

(c) We have an additive functor $\pi:  \C\rightarrow \uC$, given by 
$(M\stackrel{\alpha}{\rightarrow} N)
\mapsto (M\stackrel{[\alpha]}{\rightarrow} N)$, where $[\alpha]:=\alpha +\mathrm{Fact}_{\Q(\E)}(M,N)$. Let $M\in \C$. Since $\C$ is a 
KS-category, it follows that $M=\oplus_{i=1}^{n} M_i$, with $\End_{\C}(M_i)$ local for all $i=1,\cdots,n$. Since 
$\pi:  \C\rightarrow \uC$ is an additive functor, the same decomposition  $M=\oplus_{i=1}^{n} M_i$ holds in $\uC,$ but
by (a), the direct summand $M_i$ is zero in $\uC$ if and only if $M_i\in \Q(\E).$ Hence, without lost of 
generality, we may assume that  $M=\oplus_{i=1}^{n} M_i$, with $M_i\not\in\Q(\E)$ for all $i$. Then, by Lemma \ref{Primerlema} (b), 
we get $$\frac{\uEnd (M_i)}{\rad(\uEnd(M_i))}=
\frac{\End_{\C}(M_i) / \mathrm{Fact}_{\Q(\E)}(M_i,M_i)}{\rad(\End_{\C}(M_i))/\mathrm{Fact}_{\Q(\E)}(M_i,M_i)}\cong \frac{\End_{\C}(M_i)}{\rad(\End_{\C}(M_i))},$$
and hence the ring $\uEnd(M_i)$ is local for any $i=1,\cdots,n$.
\end{dem}

Let $f:M\rightarrow M'$ be a morphism in $\C$. Note that, the morphism $f$ can be lifted to  morphisms $f_0$ and $f_1$ such that the following $\E$-exact diagram in $\C$ commutes
$$\xymatrix@1{
\ar[d]^{f_1}  \Omega_{\E,P}(M) \ar[r]^{i} & \ar[d]^{f_0}  P \ar[r]^{\qquad\mu\quad} & M \ar[d]^{f}\\
\Omega_{\E,P'}(M') \ar[r]^{i'} & P' \ar[r]^{\qquad\mu'\quad} & M'.}$$
By using the lifted morphism $f_1,$ associated with $f,$ we can define the following functor.

\begin{pro}\label{syzygyaditivo}
The correspondence $\Omega_{\E}: \uC \rightarrow \uC$ given by 
$$(M \stackrel{[f]}{\rightarrow} M')\mapsto (\Omega_{\E,P}(M)\stackrel{[f_1]}{\rightarrow} \Omega_{\E,P'}(M')),$$  where $[f_1]:=f_1 +\mathrm{Fact}_{\Q(\E)}$ 
is a well defined additive functor.
\end{pro}
\begin{dem}
Let $f_1'$ and $f_0'$ be another lifted morphisms of $f.$ Then we have the following commutative diagram
$$\xymatrix@1{
\ar[d]^{g_1}  \Omega_{\E,P}(M) \ar[r]^{{\quad}i} & \ar[d]^{g_0}  P \ar[r]^{\qquad\mu\quad} & M \ar[d]^{0}\\
\Omega_{\E,P'}(M') \ar[r]^{{\quad}i'} & P' \ar[r]^{\qquad\mu'\quad} & M',}$$
where $g_i:=f_i-f'_i$, for $i=0,1$. Then, there exists $t: P\rightarrow \Omega_{\E,P'}(M')$ such that $i'\circ t=g_0$. Now
$i'\circ t\circ i =g_0\circ i=i'\circ g_1$ and since $i'$ is an admissible monic, we get that 
$g_1=t\circ i$. Therefore $f_1-f'_1=t\circ i\in\Fact_{\Q(\E)}$ and hence 
$\Omega_{\E}([f])= [f_1]$ is well defined. It can be seen that such a correspondence gives an additive functor.
\end{dem}
\medskip

\section{The relative Igusa-Todorov functions $\Phi_\E$ and $\Psi_\E$}

In this section, we develop the general theory of $\E$-IT functions defined over an exact $ IT$-context 
$(\C,\E).$ In order to do that, we start with some general notions.

We say that an additive category $\C$ satisfies the {\bf Fitting's property} if for any $X\in\C$ and $f\in\End_\C(X)$ there exists a decomposition $X=Y\oplus Z$ such that 
$f=\begin{pmatrix}1 & 0\\0 &\alpha\end{pmatrix}$ 
and $\alpha:Z\to Z$ is nilpotent. 
\

An abelian category $\C$ satisfies the {\bf bi-chain condition} if 
for any sequence of morphisms $X_n\xrightarrow{\alpha_n}X_{n+1}\xrightarrow{\beta_n}X_n,$ with $n\geq 0,$ $\alpha_n$ a monomorphism and $\beta_n$ an epimorphism (for any $n$) there exists $n_0$ such that 
$\alpha_n$ and $\beta_n$ are isomorphisms for any $n\geq n_0.$
\

Let $k$ be a commutative ring. A $k$-category $\C$  is {\bf Hom-finite} if the $k$-module $\Hom_\C(A,B)$ is of finite length for any $A,B\in\C.$ 
\

A category $\C$ is said to be a {\bf{length-category}}  if $\C$ is abelian, skeletally small and every object of $\C$ has a 
finite composition series. It can be seen, that any  length-category is a KS-category \cite{G}.

\begin{defi}\label{IT-context} An {\bf exact $ IT$-context} is an exact skeletally small Krull-Schmidt category $(\C,\E),$ with enough 
$\E$-projectives. Moreover, if $\C$ is abelian and $\E$ is the class of all exact sequences, we say that $\C$ is an {\bf abelian $ IT$-context}. An exact (abelian) $ IT$-context is said to be {\bf strong} if $\C$ satisfies the Fitting's property.
\end{defi}

An {\bf exact $\mathrm{IT}^{op}$-context} is an exact category $(\C,\E)$ such that the opposite 
exact category $(\C^{op},\E^{op})$ is an exact $ IT$-context. That is, $(\C,\E)$ is an exact skeletally small Krull-Schmidt category with enough 
$\E$-injectives. Note that any result obtained for exact $ IT$-contexts can be translated for exact $\mathrm{IT}^{op}$-contexts.

The following remark is rather useful to prove that some categories are  strong abelian $ IT$-contexts. The reader can provide a proof of this fact by using Section 5 in \cite{Krause}.

\begin{remark} Let $\C$ be a skeletally small abelian category with enough projectives. If
any of the following conditions hold, then $\C$ is a strong abelian $ IT$-context.
\begin{itemize}
\item[(a)] $\C$ is a Hom-finite $k$-category.
\item[(b)] $\C$ is a length-category.
\item[(c)] $\C$ satisfies bi-chain condition.
\end{itemize}
\end{remark}

From an abelian $ IT$-context, using Proposition \ref{projExcat}, we can  produce an exact $ IT$-context as follows.

\begin{remark}\label{XITCon} Let $\C$ be an abelian $ IT$-context. If $\X\subseteq\C$ is a precovering generator class in $\C,$ 
then the pair $(\C,\E_F),$ for $F:=F_\X,$ is an exact $ IT$-context.
\end{remark}

Let $\C$ be an abelian $\mathrm{IT}^{op}$-context and $\Y\subseteq\C$ be a preenveloping 
cogenerator in $\C.$ In this case, for $G:=F^{\Y},$ the pair $(\C,\E_G)$ is an exact $\mathrm{IT}^{op}$-context.

In this section, we fix an exact $ IT$-context  $(\C,\E).$ Generalising K. Igusa and G. Todorov in \cite{IT}, we introduce and develop the notion of the relative Igusa-Todorov functions $\Phi_{\E}$ and $\Psi_{\E}$ attached to the exact $ IT$-context  $(\C,\E).$ In the case of the $ IT$-context $(\C,\E_F)$ 
for some $F:=F_\X,$  we set $\Phi_{\X}:=\Phi_{\E_F}$ and $\Psi_{\X}:=\Psi_{\E_F}$ and call them $\X$-IT functions. 
\

Let $(\C,\E)$ be an exact $\mathrm{IT}^{op}$-context. In this case, the corresponding Igusa-Todorov functions will be denoted as  $\Phi^{\E}$ and $\Psi^{\E}.$ If $\E=\E_G,$ 
for some $G=F^\Y,$ we just write $\Phi^{\Y}$ and $\Psi^{\Y}.$

\begin{defi} Let $(\C,\E)$ be an exact $ IT$-context. We denote by $\K_{\E}(\C)$  the quotient of the free abelian group generated by the set of iso-classes $\{[M]\ :\ M\in\C\}$
modulo the relations:
\begin{itemize}
\item[(i)]
$[N]-[S]-[T]$ if $N\cong S\oplus T$, and
\item[(ii)]
$[P]$ if $P\in \Q(\E).$
\end{itemize}
\end{defi}

By  Proposition \ref{indenelcociente}, we have that $\K_{\E}(\C)$ is the free abelian group generated by the objects in 
the set  $\ind\,(\C)-\Q(\E)$. Furthermore, by Proposition \ref{syzygyaditivo}, the additive functor $\Omega_{\E}: \uC\rightarrow \uC$  gives rise 
 to a morphism $\Omega_{\E}: \K_{\E}(\C)\rightarrow \K_{\E}(\C)$ of abelian groups, given by
$\Omega_{\E}([M]):=[\Omega_{\E}(M)]$ for any $M\in\C.$

For each object $M\in \C$, we denote by $\langle M\rangle$  the $\Z$-submodule of $\K_{\E}(\C)$, generated by the indecomposable direct summands of $M$ in the $\E$-stable category $\uC$.

\begin{lem}\label{Fitting} \emph{(Fitting's Lemma)}
Let $R$ be a noetherian ring. Consider a left $R$-module  $M$ and $f\in \End_R(M)$. Then, for any finitely generated $R$-submodule $X$ of $M$, there is a non-negative integer 
$$\eta_f(X) :=\min\{ k\in \N\,:\, f|_{f^{m}(X)}: f^{m}(X)\stackrel{\sim}{\rightarrow} f^{m+1}(X),\ \forall m\geq k\}.$$ 
Furthermore, for any $R$-submodule $Y$ of $X$, we have that $\eta_f(Y)\leq \eta_f(X)$.
\end{lem}

Now, generalising K. Igusa and G. Todorov in \cite{IT}, for any exact $ IT$-context $(\C,\E),$ we introduce the $\E$-relative Igusa-Todorov functions.

\begin{defi} Let $(\C,\E)$ be an exact $ IT$-context. The $\E$-IT functions $\Phi_{\E},\,\Psi_{\E}: \Obj(\C)\rightarrow \N$ are defined, by using Fitting's Lemma, as follows
\begin{itemize}
\item $\Phi_\E(M) := \eta_{\Omega_{\E}}(\langle M\rangle),$ and
\vspace{0.2cm}
\item $\Psi_\E(M):=\Phi_\E(M)+\mathrm{fpd}_\E(\add(\Omega_{\E}^{\Phi_\E(M)}(M))).$
\end{itemize}
\end{defi}

\begin{remark} Let $\C$ be an abelian $ IT$-context and consider $F:=F_{\Q(\C)}.$ In this case, the subfunctor $F$ coincides with $\Ext^1_\C(-,-).$ Hence, the relative projective dimension $\pd_F(M)$ of $M$ is just the 
ordinary one $\pd(M).$ So, for the $\Q(\C)$-IT functions, we set $\Phi:=\Phi_{\Q(\C)}$ and $\Psi:=\Psi_{\Q(\C)}.$ Note that, for $\C=\modu\,(\Lambda)$ and $\Lambda$ an artin algebra, the 
$\Q(\C)$-IT functions coincide with those introduced in \cite{IT}.  
\end{remark}

In what follows, we generalise (for the relative ones) some basic properties of the Igusa-Todorov functions  given in \cite{IT, HLM}. Some 
of  them can be proven in a similar way, but for completeness, we include a proof. Before doing that, we start with the following remark and lemma.

\begin{remark}\label{Huard}
For any $M\in \C$, it follows that $\Omega_{\E}(\langle M\rangle)\subseteq \langle\Omega_{\E}(M)\rangle$.
\end{remark}

\begin{lem}\label{auxlema}  For any $M\in\C$ and $k\in\N,$  we have that 
$$\langle\Omega_{\E}^k(M)\rangle=0\;\Leftrightarrow\;\Omega_{\E}^k(\langle M\rangle)=0$$
\end{lem}
\begin{dem} $(\Rightarrow)$ Let $\langle\Omega_{\E}^k(M)\rangle=0.$ Then, by Remark \ref{Huard}, it follows that 
$\Omega_{\E}^k(\langle M\rangle)=0.$
\

$(\Leftarrow)$ Assume that $\Omega_{\E}^k(\langle M\rangle)=0.$ By Proposition \ref{indenelcociente}, we have a decomposition 
$M=\oplus_{i=1}^n\,M_i^{a_i}$ in the stable category $\uC,$ where each $M_i\in\ind\,(\uC)$ and are pairwise non-isomorphic 
 in $\uC.$ Since, for each $i,$ we have 
that $\Omega_{\E}^k(\langle M_i\rangle)\subseteq\Omega_{\E}^k(\langle M\rangle),$ we get that $\Omega_{\E}^k( [M_i])=0.$ 
 Thus $\Omega_{\E}^k(M_i)=0$ in $\uC$ for all $i,$ and so, $\Omega_{\E}^k(M)=0$ in $\uC.$ Hence, we get 
 that  $\langle\Omega_{\E}^k(M)\rangle=0.$
\end{dem}

\begin{pro}\label{basicasdefi}
For any $M\in \C$, the following statements hold true for an exact $ IT$-context $(\C,\E).$
\begin{itemize}
\item[(a)] If $\pd_\E(M)<\infty$ then $\Phi_\E(M)=\pd_\E(M)$.
\item[(b)] If $\pd_\E(M)=\infty$ and $M\in\ind\,(\C),$ then $\Phi_\E(M)=0$.
\item[(c)]
$\Phi_\E(M)\leq \Phi_\E(M\oplus N)$, for all $N\in \C$.
\end{itemize}
\end{pro}
\begin{dem}
\begin{itemize}
\item[(a)] Let $n:=\pd_\E(M)<\infty$. By Lemma \ref{primeraI}, we get that $\pd_\E(M)=\Oedim(M)$. If $n=0$ then $M\in\Q(\E)$ and so $\langle M\rangle = 0$, and this implies $\Phi_\E(M)=0$. 
\

Let $n\geq 1$. Then we have an $\E$-acyclic complex
$$\Omega_{\E}^n(M)\rightarrowtail P_{n-1} \longrightarrow\cdots \rightarrow
P_1 \rightarrow P_0 \twoheadrightarrow M,$$
\noindent where $\Omega_{\E}^n(M)\in \Q(\E)$ and $\Omega_{\E}^i(M)\rightarrowtail P_{i-1}\twoheadrightarrow\Omega_{\E}^{i-1}(M)$ is $\E$-exact such that $\Omega_{\E}^i(M)\not\in \Q(\E),$ for all $0\leq i\leq n-1$. Note that 
$\langle \Omega_{\E}^k(M)\rangle=0$ for all $k\geq n,$ and it follows from Lemma \ref{auxlema} that $\Omega_{\E}^k(\langle M\rangle)=0$, for any $k\geq n$.
So we obtain that $\Omega_{\E}: \Omega_{\E}^k(\langle M\rangle) \stackrel{\sim}{\rightarrow}  \Omega_{\E}^{k+1}(\langle M\rangle)$, for all $k\geq n$.

Suppose that $\Omega_{\E}^{n-1}(\langle M\rangle)=0$. Then 
by Lemma \ref{auxlema} it follows that $\Omega_{\E}^{n-1}(M) \in \Q(\E)$, which is a contradiction. Therefore $\Omega_{\E}^{n-1}(\langle M\rangle)\neq 0$ which means that
$\Omega_{\X}: \Omega_{\E}^{n-1}(\langle M\rangle)\rightarrow \Omega_{\E}^{n}(\langle M\rangle)$ cannot be an isomorphism, proving that $\Phi_\E(M)=n$.
\item[(b)] Let $\pd_\E(M)=\infty$ and $M\in \ind\,(\C)$. By Lemma \ref{primeraI}, we get that $\Oedim(M)=\infty$ and then
$\Omega _{\E}^i(M)\not \in \Q(\E)$, for all $i\in\N$. Since $M\in \ind\,(\C)$, the group $\langle M\rangle$ is free of rank one.  On the 
other hand, the fact that $\Omega _{\E}^i(M)\not \in \Q(\E)$, means that $\Omega _{\E}^i(\langle M\rangle )$ is free of rank one, for all $i\in \N$. Therefore $\Phi_\E(M)=0$.
\item[(c)]
It follows from Fitting's Lemma since $\langle M\rangle \subseteq \langle M\oplus N\rangle$.
\end{itemize}
\end{dem}

We generalise further properties of $\Phi_\E$ from \cite{IT,HLM}. 
\begin{pro}\label{propsfi} The following statements hold true for an exact $ IT$-context $(\C,\E).$
\begin{itemize}
\item[(a)] If $\add(M)=\add(N)$ then $\Phi_\E(M)=\Phi_\E(N)$.
\item[(b)] $\Phi_\E(M\oplus P)=\Phi_\E(M)$, for all $M\in\C$ and $P\in \Q(\E)$.
\item[(c)] $\Phi_\E(M)\leq \Phi_\E(\Omega_{\E}(M))+1$, for all $M\in \C$.
\end{itemize}
\end{pro}
\begin{dem}
\begin{itemize}
\item[(a)] If $\add(M)=\add(N)$ then $\langle M\rangle = \langle N \rangle$, and hence, $\Phi_\E(M)=\Phi_\E(N)$.
\item[(b)] Let $M\in \C$ and $P\in \Q(\E)$. Therefore $\langle M\rangle = \langle M\oplus P\rangle$ and then $\Phi_\E(M\oplus P)=\Phi_\E(M)$.
\item[(c)] By the Remark \ref{Huard}, we know that $\Omega_{\E}(\langle M\rangle)\subseteq \langle \Omega_{\E}(M)\rangle.$ So the 
 same proof given in \cite[Lemma 3.4]{HLM} works, by replacing the ordinary syzygy $\Omega$ by the relative one $\Omega_{\E}.$ 
\end{itemize}
\end{dem}

\begin{pro}\label{propsi}  The following statements hold true for an exact $ IT$-context $(\C,\E).$
\begin{itemize}
\item[(a)] If $\pd_\E(M)< \infty$ then $\Psi_\E(M)=\Phi_\E(M)=\pd_\E(M).$
\item[(b)] If $\add(M)=\add(N)$ then $\Psi_\E(M)=\Psi_\E(N).$
\item[(c)] If $\pd_\E(M)=\infty$ and $M\in\ind(\C),$ then $\Psi_\E(M)=0.$
\item[(d)] $\Psi_\E(M\oplus P)=\Psi_\E(M)$ for all $M\in\C$ and $P\in \Q(\E).$
\item[(e)] Let $Z\mid \Omega_{\E}^n(M),$ where $n\leq\Phi_\E(M)$  and $\pd_\E(Z)<\infty$. Then 
$$\pd_\E(Z)+n\leq \Psi_\E(M).$$
\item[(f)] $\Psi_\E(M)\leq \Psi_\E(M\oplus N)$ for any $M,N\in\C.$
\item[(g)] $\Psi_\E(M)\leq \Psi_\E(\Omega_{\E}(M))+1$ for all $M\in \C.$
\end{itemize}
\end{pro}
\begin{dem}
\begin{itemize}
\item[(a)]
Let $n=\pd_\E(M)<\infty$. Then, by Proposition \ref{basicasdefi} (a) and Lemma \ref{primeraI}, we get
$n=\Phi_\E(M)=\Oedim(M)$. Therefore $\Omega_{\E}^{\Phi_\E(M)}(M)\in \Q(\E)$ and then
$\add(\Omega_{\E}^{\Phi_\E(M)}(M))\subseteq \Q(\E)$, obtaining that $\mathrm{fpd_\E(\add(\Omega_{\E}^{\Phi_\E(M)}(M)))}=0.$ Thus 
$\Psi_\E(M)=\Phi_\E(M),$ proving (a).
\item[(b)] By  Proposition \ref{propsfi} (a), we know that $\Phi_\E(M)=\Phi_\E(N): = \alpha$. Hence the result follows from the fact that
$\add(\Omega_{\E}^{\alpha}(M))= \add(\Omega_{\E}^{\alpha}(N))$, since $\add(M)= \add(N)$.
\item[(c)] From  Proposition \ref{basicasdefi} (b), we know that $\Phi_{\E}(M)=0$. Furthermore, using the fact that $M\in\ind(\C),$ we get 
that $\add(\Omega_{\E}^{\Phi_{\E}(M)}(M))=\add(M)=\{M^k\;\text{with}\;k\in\N\}$ and then 
$\mathrm{fpd}_\E(\add(\Omega_{\E}^{\Phi_{\E}(M)}(M)))=0.$ Hence,  this shows that $\Psi_\E(M)=0$.
\item[(d)] It follows from Proposition \ref{propsfi} (b) and the fact that $\Omega_{\E}\mathrm{dim}\,(P)=0$, $\forall P\in\Q(\E).$ 
\item[(e)] Let $k:=\Phi_\E(M)-n$. Then $\Omega_{\E}^k(Z)\mid\Omega_{\E}^{\Phi_\E(M)}(M)$ and $\pd_\E(\Omega_{\E}^{k}(Z))< \infty,$  
 since $\mathrm{pd}_\E(Z)< \infty$ (see Lemma \ref{primeraI}). So, we get that 
 $$(*)\quad\mathrm{pd}_\E(\Omega_{\E}^{k}(Z))+\Phi_\E(M)\leq \Psi_\E(M).$$
 On the other hand, it can be seen that 
 $$\mathrm{pd}_\E(Z)+n\leq\mathrm{pd}_\E(\Omega_{\E}^{k}(Z))+k+n=\mathrm{pd}_\E(\Omega_{\E}^{k}(Z))+\Phi_\E(M).$$
 Therefore, from $(*),$ we conclude that $\pd_\E(Z)+n\leq \Psi_\E(M).$
\item[(f)] Let $A,B\in\C.$ By Proposition \ref{basicasdefi} (c), we know that $\Phi_\E(A)\leq \Phi_\E(A\oplus B).$ Hence, (f) follows 
from (e)  by taking $M:=A\oplus B,$ $n:=\Phi_\E(A)$ and $Z|\Omega_{\E}^n(A)$ such that 
$\pd_\E(Z)=\mathrm{fpd}_\E(\add(\Omega_{\E}^n(A))).$
\item[(g)] Let $M\in \C.$ By Proposition \ref{propsfi} (c), we know that $\Phi_\E(M)\leq \Phi_\E(\Omega_{\E}(M))+1.$ Take 
$n:=\Phi_\E(M)-1\leq \Phi_\E(\Omega_{\E}(M))$ and let $Z\mid\Omega_{\E}^n(\Omega_{\E}(M))$ be such that 
$$\pd_\E(Z)=\mathrm{fpd}_\E(\add(\Omega_{\E}^{\Phi_\E(M)}(M)).$$
Then,  by (e), we get that $\pd_\E(Z)+n \leq \Psi_\E(\Omega_{\E}(M)).$ Therefore
$\Psi_\E(M)=\pd_\E(Z) + \Phi_\E(M)\leq \Psi_\E(\Omega_{\E}(M))+1$.
\end{itemize}
\end{dem}

Before proving the generalisation of the main property \cite[Theorem 4]{IT} for the  $\E$-IT function $\Psi_\E,$ we state 
and prove the following lemma.
 
\begin{lem}\label{syzygysecs}
Let $A\rightarrowtail B\twoheadrightarrow C$ be an $\E$-exact sequence. Then, for each $n\in \N$, there is an $\E$-exact sequence:
$\Omega_{\E}^n(A)\rightarrowtail \Omega_{\E}^n(B)\oplus P\twoheadrightarrow \Omega_{\E}^n(C)$, where $P\in \Q(\E).$
\end{lem}
\begin{dem}
Since, for each $M\in\C,$ we can construct an $\E$-projective resolution, by applying $n$ times Lemma \ref{horseshoe} and Lemma \ref{schanuelL}, we get the result.
\end{dem}

The following theorem is the relative version of the result given, by K. Igusa and G. Todorov, in \cite[Theorem 4]{IT}. 

\begin{teo}\label{desigualdad}
Let $(\C,\E)$ be a strong exact $ IT$-context. Then, for any $\E$-exact sequence $A\rightarrowtail B\twoheadrightarrow C$   such that $\pd_\E(C)<\infty,$ we have that $$\pd_\E(C)\leq \Psi_\E(A\oplus B)+1.$$
\end{teo}
\begin{dem} Let $A\rightarrowtail B\twoheadrightarrow C$ be an $\E$-exact sequence  such that $\pd_\E(C)<\infty.$ Take $r:=\pd_\E(C)=\Oedim(C)$. By Lemma \ref{syzygysecs}, we get an $\E$-exact sequence 
$\Omega_{\E}^r(A)\rightarrowtail\Omega_{\E}^r(B)\oplus P \twoheadrightarrow \Omega_{\E}^r(C)$ with $P\in\Q(\E),$ which splits since 
$\Omega_{\E}^r(C)\in \Q(\E).$  Hence $[\Omega_{\E}^r(A)] = [\Omega_{\E}^r(B)]$ and so there  exists the natural number 
$$n:=\min\{j\in\N\ :\ [\Omega_{\E}^j(A)] = [\Omega_{\E}^j(B)] \}.$$
Note that $n\leq \pd_\E(C).$ Furthermore,  since $[A],\ [B]\in \langle A\oplus B\rangle,$  $[\Omega_{\E}^n(A)]=[\Omega_{\E}^n(B)]$ 
and $[\Omega_{\E}^{n-1}(A)] \neq [\Omega_{\E}^{n-1}(B)]$, it follows that $n\leq  \Phi_\E(A\oplus B)$. By Lemma \ref{syzygysecs} we 
 get an $\E$-exact sequence $\Omega_{\E}^n(A)\stackrel{t}{\rightarrowtail} \Omega_{\E}^n(B)\oplus P' 
 \stackrel{g}{\twoheadrightarrow} \Omega_{\E}^n(C)$, for some $P'\in \Q(\E)$. Using that 
 $[\Omega_{\E}^n(A)]=[\Omega_{\E}^n(B)]$,  we can obtain decompositions
$\Omega_{\E}^n(A)=T\oplus Q'$ and $\Omega_{\E}^n(B)\oplus P' =T\oplus Q$, with $Q, Q'\in \Q(\E)$. Thus, we can see the previous $\E$-exact sequence as:
\begin{equation}\label{trebol}
\quad \quad T\oplus Q' \stackrel{t}{\rightarrowtail} T\oplus Q \stackrel{g}{\twoheadrightarrow} \Omega_{\E}^n(C),
\end{equation} 
where $t=\left(
                                            \begin{array}{cc}
                                              f & * \\
                                              * & * \\
                                            \end{array}
                                          \right)$ with $f\in \End_{\C}(T)$. 
Since $\C$ satisfies Fitting's property, there is a decomposition $T=T_1\oplus T_2$ such that $f=\left(
                                            \begin{array}{cc}
                                              1 & 0 \\
                                              0 & \alpha \\
                                            \end{array}
                                          \right)$ and $\alpha:T_2\rightarrow T_2$ is nilpotent.
Let $M\in\C$. By applying the functor $\Hom_{\C}(-,M)$ to the $\E$-exact sequence (\ref{trebol}), and using that 
$Q,Q' \in \Q(\E)$, we get the long exact sequence  (for $k\geq 1$)
\begin{equation}\label{corazon}\quad \xymatrix{\Ext^k_\E(T,M) \ar[r]^{\gamma_k} & \Ext^k_\E(T,M) \ar[r]^{\delta_k\quad} 
 & \Ext^{k+1}_\E(\Omega_{\E}^n(C),M)\ar[dll]_{\lambda_{k+1}}\\ 
\Ext^{k+1}_\E(T,M) \ar[r]_{\quad\gamma_{k+1}\quad} & \Ext^{k+1}_\E(T,M), &
}\end{equation}
where $\gamma_k =\left(\begin{array}{cc}
                                              1 & 0 \\
                                              0 & \beta_k \\
                                            \end{array}
                                          \right)$, $\beta_k=\Ext^k_\E(\alpha, M)$ and $\lambda_{k+1}=\Ext^{k+1}_\E(g,M)$.
Note that $\beta_k$ is nilpotent, for all $k$, since $\alpha$ is nilpotent. Using that 
$$\pd_\E(\Omega_{\E}^n(C))= \pd_\E(C)-n=r-n,$$  we get  that 
\begin{equation}\label{estrella}\quad \Ext^{j}_\E(\Omega_{\E}^n(C),M)=0\quad\text{for all}\; j>r-n.
\end{equation} 
\indent
We assert now that $\pd_\E(T_2)<\infty.$ Indeed, by $(\ref{corazon})$ and $(\ref{estrella}),$ we conclude  that $\gamma_k: \Ext^k_\E(T,M) \rightarrow \Ext^k_\E(T,M)$ 
is an epimorphism for all $k>r-n-1$. Since $\gamma_k =\left(\begin{array}{cc}
                                              1 & 0 \\
                                              0 & \beta_k \\
                                            \end{array}
                                          \right)$, 
we get that  $\beta_k: \Ext_\E^k(T_2,M) \rightarrow \Ext_\E^k(T_2,M)$ is a nilpotent epimorphism for all $k>r-n-1$. Therefore $\Ext_\E^k(T_2,M)=0$, for all $k>r-n-1$ and any $M\in \C$, proving that $\pd_\E(T_2)<\infty$.

Now, we see that $\pd_\E(\Omega_{\E}^n(C))\leq \pd_\E(T_2)+1.$ Indeed, assume that\linebreak $\Ext^{k+1}_\E(\Omega_{\E}^n(C),M) \neq 0$ for some $k\geq 1$ and $M\in\C.$ So, we have that $\delta_{k}\neq 0$ or $\lambda_{k+1}\neq 0$. Note that $\delta_{k}\neq 0$ implies 
that $\gamma_{k}$ is not an epimorphism, and then, $\beta_k$ is not an epimorphism. As a consequence we get that $\Ext_\E^k(T_2,M)\neq 0$. On the other hand, if $\lambda_{k+1}\neq 0$ then $\gamma_{k+1}$ is not an monomorphism and so $\beta_{k+1}$ is not an monomorphism. Thus, we  get that $\Ext_\E^{k+1}(T_2,M)\neq 0$. Therefore, we have seen that  $\Ext^{k+1}_\E(\Omega_{\E}^n(C),M) \neq 0$ implies that $\Ext_\E^{k}(T_2,M)\neq 0$ or $\Ext_\E^{k+1}(T_2,M)\neq 0$. Hence, by  the above statement, it follows that 
$\Ext^{k+1}_\E(\Omega_{\E}^n(C),M)=0$, for all $k\geq \pd_\E(T_2)+1$ and for any $M\in \C$. Then,  we 
conclude  that $\pd_\E(\Omega_{\E}^n(C))\leq \pd_\E(T_2)+1$. Now, using the facts that $T_2\mid\Omega_{\E}^n(A\oplus B)$, $\pd_\E(T_2)<\infty$ and $n\leq \Phi_\E(A\oplus B)$, it follows from Proposition \ref{propsi} (e) that $\pd_\E(T_2) + n\leq \Psi_\E(A\oplus B)$. Therefore 
$$\pd_\E(C)=\pd_\E(\Omega_{\E}^n(C)) + n\leq \pd_\E(T_2)+n+1\leq \Psi_\E(A\oplus B)+1,$$ 
proving the inequality.
\end{dem}

\begin{cor}Let $(\C,\E)$ be a strong exact $ IT$-context. Then, for any $\E$-exact sequence $A\rightarrowtail B\twoheadrightarrow C,$ the following statements hold.
\begin{itemize}
\item[(a)] $\pd_\E(B)<\infty\quad\Rightarrow\quad\pd_\E(B)\leq 1+\Psi_\E(A\oplus\Omega_\E(C)).$
\item[(b)] $\pd_\E(A)<\infty\quad\Rightarrow\quad\pd_\E(A)\leq 1+\Psi_\E(\Omega_\E(B\oplus C)).$
\end{itemize}
\end{cor}
\begin{dem} Let $A\rightarrowtail B\twoheadrightarrow C$ be an $\E$-exact sequence. Then, using that 
$\C$ has enough $\E$-projectives, an easy application of \cite[Proposition 2.12]{Buhler}, gives us the existence 
of two $\E$-exact sequences $\Omega_\E(C)\rightarrowtail A\oplus P\twoheadrightarrow B$ and 
$\Omega_\E(B)\rightarrowtail \Omega_\E(C)\oplus Q\twoheadrightarrow A\oplus P,$ for some 
$P,Q\in\Q(\E).$ Thus, from Theorem \ref{desigualdad} and Proposition \ref{propsi} (d), we get the result.
\end{dem}

\section{Relative Igusa-Todorov dimensions}

In this section we fix an exact $IT$-context $(\C,\E),$ and for an abelian skeletally small and Krull-Schmidt category $\C,$ we 
shall consider an exact $IT$-context of the form $(\C,\E_F),$ for some $F:=F_\X,$ where $\X$ is a  class of objects in $\C.$ In 
each of these cases, we have the relative $ IT$-functions: $\Phi_\E$ and $\Psi_\E.$

\begin{defi}  Let   $\Y$ be any class of objects in $\C.$ The $\E$-relative $\Phi$-dimension of $\Y$ is 
$\Fiedim\,(\Y):=\mathrm{sup}\{\Phi_\E\,(Y)\;:\; Y\in\Y\}.$ Similarly, the $\E$-relative $\Psi$-dimension of $\Y$ is 
$\Psiedim\,(\Y):=\mathrm{sup}\{\Psi_\E\,(Y)\;:\; Y\in\Y\}.$
\end{defi}

The following inequalities are easily obtained from the definition of the relative $ IT$-functions and from the main properties given in 
the Propositions \ref{basicasdefi},  \ref{propsfi}  and \ref{propsi}.

\begin{remark}\label{primeras} For any class $\Y$ of objects, in an exact IT context $(\C,\E),$  the following inequalities hold true.
\begin{itemize}
\item[(a)] $\fpd_\E(\Y)\leq \Fiedim\,(\Y)\leq \Psiedim\,(\Y)\leq\pd_\E(\Y).$
\item[(b)] $\Psiedim\,(\Y)\leq\Fiedim\,(\Y)+\fpd_\E(\Y).$
\item[(c)] $\Psiedim\,(\Y)\leq2\,\Fiedim\,(\Y).$
\end{itemize}
\end{remark}

 In the case of the length-category $\C:=\modu\,(\Lambda),$  for some artin algebra $\Lambda,$ 
 and the exact structure $\E_F,$ for $F=F_\X,$ the $\X$-relative $\Phi$-dimension of the algebra $\Lambda$ is $\Fixdim\,(\Lambda):=\Fixdim\,(\C).$ Similarly, we 
introduce the $\X$-relative $\Psi$-dimension $\Psixdim\,(\Lambda)$ of the algebra $\Lambda.$ 
\

We recall that a Frobenius  category is an exact category $(\C,\E)$ with enough projectives and injectives such that $\I(\E)=\Q(\E).$ In this case $\uC=\overline{\C}$ and the syzygy functor 
$\Omega_\E:\uC\to\uC$ is an equivalence with inverse the co-syzygy functor 
$\Omega^{-1}_\E=\Omega_{\E^{op}}.$
Note that if the Frobenius category $(\C,\E)$ is skeletally small and Krull-Schmidt, then $(\C,\E)$ is both an exact $ IT$-context and an exact $\mathrm{IT}^{op}$-context. Doing easy calculations, we can get the following result. For a more general one, we recommend the reader to see Theorem \ref{experp1.5}.

\begin{remark}\label{RkF1} Let  $(\C,\E)$ be a Frobenius skeletally small and Krull-Schmidt category. Then 
 $\fpd_\E(\C)=\mathrm{fid}_\E(\E)=\Psiedim(\C)=\Psi^\E\mathrm{dim}(\C)=0.$
\end{remark}

In what follows, we see that certain abelian Frobenius categories can be characterised, as follows, in terms of the relative Igusa-Todorov functions.

\begin{teo}\label{ThmF1} Let $\C$ be an  skeletally small  abelian Krull-Schmidt category, with 
enough projectives and injectives. Then $\C$ is a Frobenius category if and only if 
$\Phi_{\Q(\C)}\mathrm{dim}(\C)=\Phi^{\I(\C)}\mathrm{dim}(\C)=0.$
\end{teo}
\begin{dem} The fact that Frobenius property implies the above equalities follows from Remark \ref{RkF1}.
\

Assume that $\Phi_{\Q(\C)}\mathrm{dim}(\C)=\Phi^{\I(\C)}\mathrm{dim}(\C)=0.$ We only proof that $\Q(\C)\subseteq\I(\C),$ since the other inclusion $\I(\C)\subseteq \Q(\C)$ can be obtained in a similar way.
\

Let $P\in\Q(\C).$ Suppose that $P\not\in\I(\C).$ We assert that there exists an exact sequence 
\begin{equation}\label{EF1}
\theta:\quad 0\to P\to I'\to C'\to 0
\end{equation}
such that $I'\in\I(\C)$ and $C'$ does not have non-zero projective direct summands. Indeed, since 
$\C$ has enough injectives there is an exact sequence $0\to P\to I\stackrel{\beta}{\to}C\to 0$ with 
$I\in\I(\C).$ Note that $C\not\in\Q(\C),$ since $P\not\in\I(\C).$ By the Krull-Schmidt property on $\C,$ we get 
a decomposition $C=P'\oplus C',$ where $P'\in\Q(\C)$ and  $C'$ does not have non-zero projective direct summands. Let $i_1:P'\to C$ and $\pi_1:C\to P'$ be, respectively, the natural inclusion and projection defined by the coproduct $C=P'\oplus C'.$ Since $P'$ is projective, there is a morphism 
$\alpha:P'\to I$ such that $i_1=\beta\alpha$ and hence $(\pi_1\beta)\alpha=1_{P'}.$ This fact implies 
that $\alpha:P'\to I$ is an split-mono and so we get a decomposition $I=P'\oplus I'$ in such a way 
that $\beta=\begin{pmatrix}
1 & 0\\ 0 & \beta'
\end{pmatrix}:P'\oplus I'\to P'\oplus C'.$ Thus, we get that $0\to P\to I'\stackrel{\beta'}{\to}C'\to 0$ is the desired exact sequence as in (\ref{EF1}).
\

We assert now that 
\begin{equation}\label{EF2}
I'\not\simeq C' \text{ in the stable category }\uC.
\end{equation}
Indeed, suppose that $I'\simeq C'$  in the stable category $\uC.$ Then, from \cite[Proposition 1.44]{AB}, there are projective objects $Q$ and $Q'$ such that $I'\oplus Q\simeq C'\oplus Q'$ in $\C.$ Hence, using that $\C$ is a Krull-Schmidt category, it follows that $C'|I'$ and thus $C'$ is injective. Therefore 
the exact sequence given in (\ref{EF1}) says us that $\id\,(P)=1,$ contradicting that $\Phi^{\I(\C)}(P)=0;$ proving then the assertion above.
\

Now, we apply Horseshoe's Lemma (see Lemma \ref{horseshoe}) to the exact sequence in (\ref{EF1}). Then, we get that $\Omega(I')\simeq\Omega(C')$ in the stable category $\uC.$ Therefore, by using that 
$I'\not\simeq C'$ in $\uC,$ it follows that $\Phi_{\Q(\C)}(I'\oplus C')\geq 1,$ contradicting the 
fact that $\Phi_{\Q(\C)}\mathrm{dim}(\C)=0.$ So, we have that $P$ has to be injective.
\end{dem}

Let now $\C$ be an abelian category. For any classes of objects $\X\subseteq\X'\subseteq\C,$  we get exact categories $(\C,\E_F)$ and $(\C,\E_{F'}),$ with $F:=F_{\X}$ and $F':=F_{\X'}.$  Observe that $F'(A,B)\subseteq F(A,B)$ for any $A,B\in\C,$ and thus 
 $\E_{F'}\subseteq \E_F.$ Therefore, it is quite natural to consider exact categories $(\C,\E')$ and $(\C,\E)$ such  that 
 $\E'\subseteq\E.$ As we can se below, these special pairs of exact categories  have nice homological properties when we are interested in computing relative projective dimensions. The following result illustrates this point of view. 

\begin{pro}\label{clasico3} Let $(\C,\E')$ and $(\C,\E)$ be exact categories, with enough $\E$-projectives and $\E'$-projectives, such that $\E'\subseteq\E$ and $\Q(\E')\subseteq\mathcal{P}^{<\infty}_{\E}(\C).$ Then, the following statements hold true.
\begin{itemize}
\item[(a)] If $\pd_{\E'}(\C)$ is finite, then $\mathcal{P}^{<\infty}_{\E}(\C)=\C.$
\item[(b)] $\pd_\E(\C)\leq \pd_{\E'}(\C)+\pd_\E(\Q(\E')).$
\end{itemize}
\end{pro}
\begin{dem}  Assume that $n:=\pd_{\E'}(\C)$ is finite and let $Z\in\C.$ Since $\pd_{\E'}(Z)\leq n,$  we get an $\E'$-acyclic complex 
$$\eta:\quad \Omega_{\E'}^k(Z)\rightarrowtail P'_{k-1}\to\cdots\to P'_0\twoheadrightarrow Z,$$
where $k\leq n,$ $P'_k:=\Omega_{\E'}^k(Z)\in\Q(\E'),$ $P'_i\in\Q(\E'),$  and 
$\Omega_{\E'}^{i}(Z)\rightarrowtail P'_{i-1}\twoheadrightarrow\Omega_{\E'}^{i-1}(Z)$ is $\E'$-exact for all $1\leq i\leq k.$ Note that $\eta$ is 
 also an $\E$-acyclic complex. Using now that 
$\Q(\E')\subseteq \mathcal{P}^{<\infty}_{\E}(\C),$ it can be seen that $\Omega_{\E'}^{t}(Z)\in \mathcal{P}^{<\infty}_{\E}(\C)$ for any 
$t\in\N.$ Furthermore, from the $\E$-acyclic complex $\eta,$ it follows that 
$$\pd_\E(Z)\leq k+\max_{0\leq i\leq k}\,\pd_\E(P'_i).$$

Thus, we get $\pd_\E(Z)\leq n+\pd_\E(\Q(\E')),$ proving the result.
\end{dem}

\begin{defi}\label{DefIT-triple} A {\bf triple $ IT$-context} is a triple $(\C,\E',\E)$ such that $(\C,\E')$ and $(\C,\E)$
 are exact $ IT$-contexts with $\E'\subseteq\E.$ Moreover, in case $\C$ satisfies also the Fitting's property, we 
 say that $(\C,\E',\E)$ is a {\bf strong triple $ IT$-context}.
\end{defi}

In what follows, for given $\Y$ and $\mathcal{Z}$ classes of objects in $\C,$ we denote by $\Y\oplus\mathcal{Z}$ the 
class of objects $W$ in $\C$ such that $W=Y\oplus Z,$ for some $Y\in\Y$ and $Z\in\mathcal{Z}.$ Now, we illustrate several connections 
between triple $ IT$-contexts and relative $ IT$-functions.

\begin{teo}\label{clasico1} Let $(\C,\E',\E)$ be a strong triple $ IT$-context and $s\in\N.$ Then
$$\fpd_\E(\C)\leq 1+s+\Psiedim\,(\Q(\E')\oplus\Omega_{\E'}\Omega^{s}_\E(\mathcal{P}^{<\infty}_{\E}(\C))).$$
\end{teo} 
\begin{dem} Let $Z\in \mathcal{P}^{<\infty}_{\E}(\C).$ Since $\C$ has enough $\E$-projectives we have an $\E$-acyclic complex
$$ \Omega^{s}_\E(Z)\rightarrowtail P_{s-1}\to\cdots \to P_1\to P_0\twoheadrightarrow Z,$$
where $P_i\in\Q(\E)$ for any $0\leq i\leq s-1.$ Using now that $\C$ has enough $\E'$-projectives 
we have an $\E'$-exact sequence
$$\eta:\; \Omega_{\E'}(\Omega^{s}_\E(Z))\rightarrowtail Q\twoheadrightarrow \Omega^{s}_\E(Z),$$
with $Q\in\Q(\E').$  But $\eta$ is also an $\E$-exact sequence, and thus by Theorem \ref{desigualdad}, we get that 
$$\pd_\E(\Omega^{s}_\E(Z))\leq 1+\Psi_\E\,(Q\oplus \Omega_{\E'}\Omega^{s}_\E(Z)).$$ 
Thus, the result follows from  the above inequality, since $\pd_\E(Z)\leq \pd_\E(\Omega^{s}_\E(Z)) +s.$
\end{dem}

The above theorem has several nice consequences as can be seen in the following corollaries. 

\begin{cor}\label{clasico2} Let $(\C,\E',\E)$ be a strong triple $ IT$-context and and $s\in\N.$ Then
$$\pd_{\E'}(\Omega^{s}_\E(\mathcal{P}^{<\infty}_{\E}(\C)))\leq 1\quad\Rightarrow\quad\fpd_{\E}(\C)\leq 1+s+\Psiedim\,(\Q(\E')).$$
\end{cor}
\begin{dem}  Let $\pd_{\E'}(\Omega^{s}_\E(\mathcal{P}^{<\infty}_{\E}(\C)))\leq 1.$ In this case, we have that $\Omega_{\E'}\Omega^{s}_\E(\mathcal{P}^{<\infty}_{\E}(\C))\subseteq\Q(\E')$ 
and thus the result follows from Theorem \ref{clasico1}. 
\end{dem}

\begin{cor}\label{rkbasico1} Let $\Lambda$ be an artin algebra, and let $\X\subseteq\X'\subseteq\modu\,(\Lambda)$ be  precovering  and generator classes in $\modu\,(\Lambda).$ Then, for $F:=F_\X$ and   $F':=F_{\X'},$ we have that
$$\gldim_{F'}(\Lambda)\leq 1\;\Rightarrow\;\findim_{F}(\Lambda)\leq\Psixdim\,(\X')+1.$$
\end{cor}
\begin{dem} It is clear that $(\modu\,(\Lambda),\E_{F'}, \E_F)$ is a  strong triple $ IT$-context. Therefore the result is  a direct consequence of Corollary \ref{clasico2}, by taking $s=0.$
\end{dem}

We recall that the {\bf representation dimension} 
$\repdim\,(\Lambda),$ of an artin algebra $\Lambda,$ is the minimal global dimension $\gldim\,\End_\Lambda(M)$ 
over all $\Lambda$-modules $M,$ which are generator and cogenerator in $\modu\,(\Lambda).$ This dimension  was introduced by M. Auslander in \cite{A2}. An object $M\in\modu\,(\Lambda),$ realising the minimum above, is called Auslander generator.
 
It  was proved by O. Iyama in \cite{Iyama}, that 
$\repdim\,(\Lambda)$ is finite. It is also well known, as can be seen  in 
\cite{O},  that for  an Auslander generator $M\in\modu\,(\Lambda),$ it follows that 
$$\gldim_{G}(\Lambda)=\repdim\,(\Lambda)-2,$$
where $G:=F_{\add\,(M)}.$

As a consequence of the above corollary, we  obtain the following main result appearing in \cite{IT}. 

\begin{cor}\label{rkbasico1.1} Let $\Lambda$ be an artin algebra, and let $M$ be an Auslander generator in $\modu\,(\Lambda).$ Then 
$$\repdim\,(\Lambda)\leq 3\;\Rightarrow\;\findim\,(\Lambda)\leq\Psi\,(M)+1<\infty.$$
\end{cor}
\begin{dem} Take $\X:=\proj\,(\Lambda)$ and $\X':=\add\,(M).$ Now, by Proposition \ref{propsi} (b) and  (f), we have that 
$\Psixdim\,(\X')=\Psix\,(M).$ Thus, we get the result from Corollary \ref{rkbasico1}, since $\gldim_{F'}(\Lambda)=\repdim\,(\Lambda)-2.$
\end{dem}

\begin{remark}\label{rkbasico1.2} Let  $\Lambda$ be an artin algebra,  $M$ be an Auslander generator in $\modu\,(\Lambda)$   
 and $G:=F_{\add\,(M)}.$ Since $\gldim_{G}(\Lambda)=\repdim\,(\Lambda)-2<\infty,$ we get from Remark \ref{primeras} (a) the following  
$$\Psi_{\add\,(M)}\mathrm{dim}\,(\Lambda)=\repdim\,(\Lambda)-2.$$
Therefore, the representation dimension of an algebra is just a particular case of a relative $IT$-dimension.
\end{remark}

\begin{cor}\label{clasico4} Let $\Lambda$ be an artin algebra,  $\X$ be a functorially finite generator in $\modu\,(\Lambda)$ and
 $\X'(n):=\X\oplus\Omega^n_\X(\modu\,(\Lambda))$ for any $n\in\N.$ Then,  for all $s\in\N,$ we have 
$$\fpd_{F_\X}(\Lambda)\leq1+s+\Psi_\X\mathrm{dim}\,(\X'(n)\oplus\Omega_{\X'(n)}\Omega^{s}_\X(\Q^{<\infty}_{F_\X}(\Lambda))).$$
\end{cor}
\begin{dem} Since $\X$ is a functorially finite generator class in $\modu\,(\Lambda),$ we get by \cite[Proposition 3.7]{AS1} that 
$\X'(n)$  is a generator precovering class in $\modu\,(\Lambda).$ Hence, for 
$F:=F_\X$ and $F':=F_{\X'},$ the triple $(\modu\,(\Lambda),\E_{F'}, \E_F)$ is a  strong triple $ IT$-context. Therefore, the result follows from Theorem \ref{clasico1}.
\end{dem}

\begin{cor}\label{clasico5} Let $\Lambda$ be an artin algebra and
 $\X'(n):=\proj\,(\Lambda)\oplus\Omega^n(\modu\,(\Lambda))$ for any $n\in\N.$ Then, for any non-negative integer $s,$ we have 
$$\fpd(\Lambda)\leq1+s+\Psi\mathrm{dim}\,(\Omega^n(\modu\,(\Lambda))\oplus\Omega_{\X'(n)}\Omega^s(\Q^{<\infty}(\Lambda))).$$
\end{cor}
\begin{dem} It follows from Corollary \ref{clasico4} by taking $\X.=\proj\,(\Lambda)$ and using Proposition \ref{propsi} (d).
\end{dem}

\section{Relative syzygy functors on $ IT$-contexts}

Let $(\C,\E)$ be an exact category with enough projectives. For any $\Y\subseteq \C,$ we consider the left $\E$-perpendicular  full subcategory ${}^{{}_\E\perp}\Y$ of $\C,$ whose objects are all 
$M\in\C$ such that $\Ext_\E^i(M,Y)=0$ for any $Y\in\Y$ and $i\geq 1.$ 


\begin{teo}\label{xperp}
Let $(\C,\E)$ be an exact  category with enough projectives, which is weakly idempotent complete, and $\A:={}^{{}_\E\perp}\Q(\E).$ Then, the restriction $\Omega_\E:\underline{\A}\rightarrow\underline{\A},$ of the syzygy functor 
$\Omega_\E:\uC\rightarrow\uC,$  is full and faithful.
\end{teo}
\begin{proof} For any $C\in\C,$ we fix an admissible epic $P_0(C)\twoheadrightarrow C;$ and thus we 
get the canonical $\E$-exact sequence $\Omega_\E(C)\rightarrowtail P_0(C)\twoheadrightarrow C.$
\

We start by proving that $\Omega_\E(\A)\subseteq \A.$ Indeed, let $A\in\A$ and $P\in\Q(\E).$ Then, by 
applying the functor $\Hom_\C(-,P)$ to the canonical $\E$- exact sequence $\Omega_\E(A)\rightarrowtail P_0(A)\twoheadrightarrow A,$ we get the exact sequence
$$\Ext_\E^i(A,P)\rightarrow\Ext_\E^i(P_0(A),P)\rightarrow \Ext_\E^i(\Omega_\E(A),P)\rightarrow \Ext_\E^{i+1}(A,P)$$
Since $\Ext_\E^i(A,P)=0,$ for all $i\geq 1,$ we get that 
$$\Ext_\E^i(\Omega_\E(A),P)\simeq\Ext_\E^i(P_0(A),P)=0\quad \forall\, i\geq 1,\;\forall \,P\in\Q(\E),$$
and thus $\Omega_\E(A)\in{}^{{}_\E\perp}\Q(\E)=\A.$

Let $M,N\in \A,$ we prove now that 
$$\Omega_\E:\underline{\mathrm{Hom}}(M,N)\rightarrow \underline{\mathrm{Hom}}(\Omega_\E(M),\Omega_\E(N))$$ 
is an isomorphism. 
\

Indeed, we start by proving that this map is a monomorphism. Let $[f]:M\rightarrow N$ be such that $\Omega_\E([f])=[f_1]=0,$ 
where $f_1$ is obtained from the following commutative and $\E$-exact diagram
$$\xymatrix@R0.4cm@C=0.6cm{   
\eta:\quad  \Omega_\E(M) \ar[r]^{\alpha_1} \ar[d]^{f_1} & P_0(M)\ar[r]^{\mu_1} \ar[d]^{f_0} & M \ar[d]^{f} \\
\varepsilon:\quad  \Omega_\E(N) \ar[r]^{\alpha_2} & P_0(N)\ar[r]^{\mu_2}  & N.  }$$
 Since $[f_1]=0,$ then $f_1$ factors throughout and admissible epic $\mu: Q\twoheadrightarrow \Omega_\E(N)$ for some 
$Q\in\Q(\E).$ That is, there exists $\lambda:\Omega_\E(M)\rightarrow Q$ such that $f_1=\mu\lambda.$
\
 
In order to see that $[f]=0,$ it is enough to prove that the pull-back $\varepsilon f,$ of the $\E$-exact sequence $\varepsilon$ by the 
morphism $f,$ splits. To see that $\varepsilon f$ splits, using \cite[Proposition 3.1]{Buhler},  we factorise the morphism $(f_1,f_0,f):\eta\to\varepsilon$  through an  $\E$-exact sequence as can be seen in the following 
$\E$-exact and commutative diagram

$$\xymatrix@R0.4cm@C=0.6cm{   
\eta:  \quad \Omega_\E(M) \ar[r]^{\alpha_1} \ar[d]^{f_1} & P_0(M)\ar[r]^{\mu_1} \ar[d]^{\alpha_0} & M \ar@{=}[d]  \\
\varepsilon f:\quad  \Omega_\E(N) \ar[r] \ar@{=}[d] & T\ar[r] \ar[d]^{f'} & M \ar[d]^{f} \\
\varepsilon:\quad   \Omega_\E(N) \ar[r]^{\alpha_2} & P_0(N)\ar[r]^{\mu_2}  & N. }$$

Note that $f'\alpha_0=f_0.$ Moreover, the push-out $f_1\eta$ of the $\E$-exact sequence $\eta$ by 
the morphism $f_1$ is just 
$\varepsilon f,$ that is, $f_1\eta = \varepsilon f$ (see \cite[Proposition  2.12]{Buhler}).  Now, we consider
 the push-outs $\lambda\eta$ and $\mu(\lambda\eta).$ Since $f_1=\mu\lambda,$ we get that  
 $\mu(\lambda\eta)=\varepsilon f,$ and thus it is obtained  the following $\E$-exact and commutative 
 diagram 
$$\xymatrix@R0.4cm@C=0.6cm{   
\eta: \quad   \Omega_\E(M) \ar[r]^{\alpha_1} \ar[d]^{\lambda} & P_0(M)\ar[r]^{\mu_1} \ar[d]^{\lambda'} & M \ar@{=}[d] \\
\lambda\eta:\quad   Q{\qquad}{\quad} \ar[r]\ar[d]^{\mu} & E\ar[r] \ar[d]^{\mu'} & M\ar@{=}[d]   \\  
\varepsilon f:\quad  \Omega_\E(N) \ar[r]  & T\ar[r]  & M,  }$$
where $\mu'\lambda'=\alpha_0.$  But, we know that $\Ext_\E^{1}(M,Q)=0,$ since $Q\in\Q(\E)$ and $M\in{}^{{}_\E\perp}\Q(\E).$ In particular, 
the $\E$-exact sequence $\lambda\eta$ splits. Therefore $\varepsilon f$ splits because of the equality $\varepsilon f= \mu(\lambda\eta),$ 
proving that $[f]=0.$
\

Now, we prove that $\Omega_\E:\underline{\mathrm{Hom}}(M,N)\rightarrow \underline{\mathrm{Hom}}(\Omega_\E(M),\Omega_\E(N))$ is an epimorphism.
Let $g:\Omega_\E(M)\rightarrow \Omega_\E(N)$ be a morphism in $\C.$ Consider the following $\E$-exact diagram in $\C$
$$\xymatrix@R0.4cm@C=0.6cm{   
\eta:\quad   \Omega_\E(M) \ar[r]^{\alpha_1} \ar[d]^{g} & P_0(M)\ar[r]^{\mu_1} & M \\
\varepsilon: \quad \Omega_\E(N) \ar[r]^{\alpha_2} & P_0(N)\ar[r]^{\mu_2}  & N.}$$
Hence, by taking the 
push-out (see \cite[Proposition  2.12]{Buhler}) $g\eta$  of the $\E$-exact sequence $\eta$ by the morphism $g,$ we get the following $\E$-exact and commutative diagram
$$\xymatrix@R0.4cm@C=0.6cm{   
\eta:\quad \Omega_\E(M) \ar[r]^{\alpha_1} \ar[d]^{g} & P_0(M)\ar[r]^{\mu_1} \ar[d]^{g'} & M \ar@{=}[d]  \\
g\eta :\quad \Omega_\E(N) \ar[r] ^{\alpha'_1} & T\ar[r] ^{\mu'_1} & M. }$$
Now, we do the push-out $\alpha_1'\varepsilon$ of the $\E$-exact sequence $\varepsilon$ by the morphism $\alpha'_1.$ Then, 
from the Snake's Lemma (see \cite[Proposition 8.11]{Buhler}) we get the following $\E$-exact and commutative  diagram
$$\xymatrix@R0.9cm@C=0.8cm{  
\varepsilon:\quad  \Omega_\E(N) \ar[r]^{\alpha_2} \ar[d]^{\alpha'_1} & P_0(N)\ar[r]^{\mu_2} \ar[d]^{\alpha''_1} & N \ar@{=}[d] \\
\alpha_1'\varepsilon:\quad T{\qquad\quad} \ar[r]^{{\qquad}\lambda}\ar[d]^{\mu'_1} & E \ar[r]^{t}\ar[d]^{\mu''_1}  & N\\ 
 M\ar@{=}[r] & M  &\ &\  }$$
Then, using that $\Ext_\E^{1}(M,P_0(N))=0,$ it follows that the second 
column in the above diagram splits. That is $E=P_0(N)\oplus M,$ $\alpha''_1=\small{\begin{pmatrix}1\\ 0\end{pmatrix}}$ and $\mu''_1=(0,1).$ In particular, 
we have that $\lambda=\small{\begin{pmatrix}\lambda_1\\ \lambda_2\end{pmatrix}}$ and $t=(t_1,t_2).$ Moreover, the equalities 
$\lambda\alpha'_1=\alpha''_1\alpha_2,$  $t\alpha''_1=\mu_2,$  $\mu''_1\lambda=\mu'_1$ and $t\lambda =0$ give us that 
\begin{equation*} (*)\quad
\begin{split}
\alpha_2 & = \lambda_1\alpha'_1,\\
 0 & = \lambda_2\alpha'_1,\\
 t_1& =\mu_2,\\
 \lambda_2 & = \mu'_1,\\
 \mu_2\lambda_1+t_2\mu'_1 &= t_1\lambda_1+t_2\lambda_2=0.
\end{split}
\end{equation*}
We assert that  $t_2:M\to N$ satisfies that $\Omega_\E([-t_2])=[g].$ Indeed, we start by lifting $t_2$ and getting the following $\E$-exact and
commutative diagram
$$\xymatrix@R0.4cm@C=0.6cm{   
 \eta:\quad \Omega_\E(M) \ar[r]^{\alpha_1} \ar[d]^{t_{2,1}} & P_0(M)\ar[r]^{\mu_1}\ar[d]^{t_{2,0}} & M \ar[d]^{t_2} \\
 \varepsilon:\quad \Omega_\E(N) \ar[r]^{\alpha_2} & P_0(N)\ar[r]^{\mu_2}  & N.  }$$
In what follows, we make use of the following commutative diagram
$$\xymatrix@R0.4cm@C=0.6cm{  
 \eta:\quad \Omega_\E(M) \ar[r]^{\alpha_1} \ar[d]^g & P_0(M)\ar[r]^{\mu_1} \ar[d]^{g'} & M \ar@{=}[d] \\
 g\eta:\quad \Omega_\E(N) \ar[r]^{\alpha'_1}\ar@{=}[d] & T \ar[r]^{\mu'_1}\ar[d]^{\lambda_1}  & M\ar[d]^{-t_2} \\ 
\varepsilon:\quad \Omega_\E(N)\ar[r]^{\alpha_2} & P_0(N)\ar[r]^{\mu_2}  & N,  }$$
where the commutativity of the bottom part are given by the equalities in $(*).$
\
 
Consider the morphism $\lambda_1g'+t_{2,0}:P_0(M)\to P_0(N).$ By using the preceding two commutative diagrams,   we get 
that $\mu_2(\lambda_1g'+t_{2,0})=0.$ Thus, there exists a morphism $\delta:P_0(M)\to \Omega_\E(N)$ such that 
$\lambda_1g'+t_{2,0}=\alpha_2\delta.$ But now, we can see that 
$$ \alpha_2\delta\alpha_1 =\lambda_1g'\alpha_1+t_{2,0}\alpha_1 = \lambda_1\alpha'_1g+\alpha_2 t_{2,1} =\alpha_2(g+t_{2,1}),
$$
 and since $\alpha_2$ is an admissible monic, we get that $\delta\alpha_1=g+t_{2,1}.$ Therefore, we conclude that 
$\Omega_\E([-t_2])=-[t_{2,1}]=[-t_{2,1}]=[g-\delta\alpha_1]=[g].$
\end{proof}

\begin{cor}\label{experp1} Let $(\C,\E)$ be an exact $IT$-context, which is weakly idempotent complete. Then, the following statements hold true.
\begin{enumerate}
\item[(a)] $\Phi_\E(M)=0=\Psi_\E(M)$ $\;\forall\, M\in {}^{{}_\E\perp}\Q(\E).$
\item[(b)] $\pd_\E(M)=\infty$ if $M\in {}^{{}_\E\perp}\Q(\E)-\Q(\E).$
\end{enumerate}
\end{cor}
\begin{proof}
(a) Let $M\in {}^{{}_\E\perp}\Q(\E).$ Consider a decomposition  $M=\oplus_{i=1}^kM_i^{\alpha_i}$ of $M$ in the stable category $\uC,$ 
 where each $M_i\in\ind(\uC)$ and are pairwise non-isomorphic. Observe that 
$M_i\in {}^{{}_\E\perp}\Q(\E),$ for any $ i,$ and $\langle M\rangle=\langle [M_1],\ldots ,[M_k]\rangle.$ So we get that  rk$(\langle M\rangle)=k.$ 
Let us assume that $\Phi_\E(M)=t>0$. Then $t$ is the last natural number where the rank falls down when we apply the morphism of abelian groups
$$\Omega_\E:\Omega_\E^{t-1}\langle [M_1],\ldots , [M_k]\rangle\rightarrow \Omega_\E^t\langle[M_1],\ldots ,[M_k]\rangle.$$
Then we get 
$$ (*) \quad  \sum_{i=1}^kt_i\Omega_\E^t[M_i]=\sum_{i=1}^kr_i\Omega_\E^t[M_i]\ \ \mathrm{with}\ \ t_i,r_i\geq 0,$$
which is a new non-trivial relation (that is, there exists $i\geq 1$ such that $t_i\neq r_i$) .
Therefore $$(**) \quad \Omega_\E(\oplus_{i=1}^k\Omega_\E^{t-1}(M_i^{t_i}))\simeq \Omega_\E(\oplus_{i=1}^k\Omega_\E^{t-1}(M_i^{r_i}))\ \
\mathrm{in} \ \ \uC.$$
Since $M_i\in {}^{{}_\E\perp}\Q(\E),$ for any $ i,$ we have by Theorem \ref{xperp} and $(**)$ that  
$$\oplus_{i=1}^k\Omega_\E^{t-1}(M_i^{t_i}) \simeq \oplus_{i=1}^k\Omega_\E^{t-1}(M_i^{r_i})\;\text{ in }\;\uC.$$  
So, the relation $(*)$ was already true for the subgroup $\Omega_\E^{t-1}\langle [M_1],\ldots , [M_k]\rangle$ and we get a 
contradiction; proving  that  $\Phi_\E(M)=0$ and thus  $\Psi_\E(M)=\fpd_\E(\add\,(M)).$
\

 We assert now that  $\Psi_\E(M)=0.$ Indeed, suppose there is $N\mid M$ with $s:=\pd_\E(N)$ finite and non-zero. Then by Lemma \ref{primeraI} we get that  
$\Omega^s_\E([N])=0,$  that is, the rank drops at the moment $s,$ contradicting that this 
rank never does it, since $\Phi_\E(M)=0.$ 
\

(b) Let $M\in {}^{{}_\E\perp}\Q(\E)-\Q(\E).$ Assume that $n:=\mathrm{pd}_\E(M)$  is finite. Since
$M\not\in\Q(\E),$ we obtain that $n>0.$ On the other hand,  by Proposition \ref{basicasdefi} (a), it follows that $\Phi_\E(M)=\mathrm{pd}_\E(M)=n,$
which is a contradiction to item (a). Therefore, we get that $\pd_\E(M)=\infty.$
\end{proof}

The next result is a generalisation of the \cite[Corollary 3.17]{LM}. 

\begin{teo}\label{experp1.5} Let $(\C,\E)$ be an exact $IT$-context, which is weakly idempotent complete. Then, the following statements 
hold true.
\begin{itemize}
\item[(a)] If $n:=\id_\E(\Q(\E))$ is finite, then  
$\Fiedim\,(\Omega_\E^n(\C))=0=\Psiedim\,(\Omega_\E^n(\C)).$
\item[(b)]  $\fpd_\E(\C)\leq\Fiedim\,(\C)\leq\Psiedim\,(\C)\leq \id_\E(\Q(\E)).$
\end{itemize}
\end{teo}
\begin{proof}  
(a) Assume that $n:=\id_\E(\Q(\E))$ is finite. Let $M\in\C,$ and since $\C$ has enough $\E$-projectives, we 
have the $\E$-acyclic complex 
 $$\Omega_\E^n(M)\rightarrowtail P_{n-1}\rightarrow \cdots \rightarrow P_0\twoheadrightarrow M,$$
 where  $P_{i}\in \Q(\E)$ for all $i.$  For any $P\in\Q(\E)$ and $j\geq 1,$ we know that   
 $$\Ext_\E^j(\Omega_\E^n(M),P)\simeq \Ext_\E^{n+j}(M,P).$$ 
 Therefore,  we obtain from the above isomorphism that $\Omega^n_\E(M)\subseteq {}^{{}_\E\perp}\Q(\E),$ since $n=\id_\E(\Q(\E)).$ Thus, the 
 equalities in  (a) can be obtained from Corollary \ref{experp1}.
 \
 
 (b) Let $M\in\C.$  From Proposition  \ref{propsi} (g), we get that   
 $\Psi_\E(M) \leq n+\Psi_\E(\Omega_\E^n(M)).$ Therefore, from (a), we conclude that
$\Psiedim\,(\C)\leq \id_\E(\Q(\E)).$ Then, from Remark \ref{primeras}, the result follows.
\end{proof}

The theorem above has  a serie of consequences as can be seen below. Note that, for any abelian $IT$-context $\C$ and  a precovering generator class $\X$  in $\C,$ we get that the pair $(\C,\E_{F_\X})$ is an $IT$-context (see Remark \ref{XITCon}). Since $\C$ is abelian, it follows that $\C$ is weakly idempotent complete.

\begin{cor}\label{experp2} Let $\C$ be an abelian $IT$-context and $\X$ be a precovering generator class   in $\C.$ Then, for $F:=F_\X,$ the following statements 
hold true.
\begin{itemize}
\item[(a)] If $n:=\id_F(\X)$ is finite, then  
$\Fixdim\,(\Omega_\X^n(\C))=0=\Psixdim\,(\Omega_\X^n(\C)).$
\item[(b)]  $\fpd_F(\C)\leq\Fixdim\,(\C)\leq\Psixdim\,(\C)\leq \id_F(\X)\leq\id\,(\X).$
\end{itemize}
\end{cor}
\begin{proof} As we have seen above, the pair $(\C,\E_F)$ is an exact $IT$-context, which is idempotent complete. We have that $\id_F(\X)\leq\id\,(\X),$ since any relative $n$-exact sequence  is
 in particular an $n$-exact sequence. So the result follows from Theorem \ref{experp1.5}, because $\Q(\E_F)=\X.$
\end{proof}

\begin{cor}\label{experp2.1} Let $\C$ be an abelian $IT$-context and $\X$ be a precovering generator class   in $\C,$ such that 
$\X$ is closed under extensions and has an $\X$-injective relative cogenerator in $\X.$ Then, for $F:=F_\X,$ the following 
inequalities hold
 $$\resdim_\X(\Q^{<\infty}_F(\C))\leq\Fixdim\,(\C)\leq\Psixdim\,(\C)\leq\id\,(\X).$$
\end{cor}
\begin{proof} By definition, we have that $\fpd_F(\C)=\pd_F(\Q^{<\infty}_F(\C)).$ Then, the result follows from Proposition \ref{igualdades} (b) and Corollary \ref{experp2}.
\end{proof}

\begin{cor}\label{experp2.2} Let $(\X,\Y)$ be a complete hereditary cotorsion pair in an skeletally small and Krull-Schmidt abelian category $\C.$ Then, the following statements hold true.
\begin{itemize}
\item[(a)] If $\C$ has enough projectives then 
$$\resdim_\X(\Q^{<\infty}_{F_\X}(\C))\leq\Fixdim\,(\C)\leq\Psixdim\,(\C)\leq\id\,(\X)$$
\item[(b)] If $\C$ has enough injectives then 
$$\coresdim_\Y(\mathcal{I}^{<\infty}_{F^\Y}(\C))\leq  \Phi^{\Y}\mathrm{dim}\,(\C)\leq\Psi^\Y\mathrm{dim}\,(\C)\leq\pd\,(\Y)$$
\end{itemize}
\end{cor}
\begin{proof}  Let $\omega:=\X\cap \Y.$ 
\

(a)  Since $(\X,\Y)$ is a complete hereditary cotorsion pair, it follows that the  pair $(\X,\omega)$ satisfies the following conditions: 
$\X$  is a precovering generator class   in $\C,$ such that  $\X$ is closed under extensions and $\omega$ is an $\X$-injective relative cogenerator in $\X.$ That is, we can apply Corollay \ref{experp2.1} and so we get (a).

(b) Using  that $(\X,\Y)$ is a complete hereditary cotorsion pair, it can be seen that the pair $(\omega,\Y)$ satisfies the 
following conditions: 
$\Y$  is a preenvelopin cogenerator class   in  $\C,$ such that  $\Y$ is closed under extensions and $\omega$ is an $\Y$-projective relative generator in $\Y.$ That is, we can apply the dual of Corollay \ref{experp2.1} and so we get (b).
\end{proof}

\begin{cor} For any artin algebra $\Lambda,$ the following inequalities hold true
$$\mathrm{findim}(\Lambda)\leq\Phi\mathrm{dim}\,(\Lambda)\leq\Psi\mathrm{dim}\,(\Lambda)\leq 
\id\,({}_\Lambda\Lambda).$$
\end{cor}
\begin{proof} Let $\Lambda$ be an artin algebra. Consider the abelian $ IT$-context
$\C:=\modu\,(\Lambda)$ and $\X:=\proj\,(\Lambda).$ In this case the subfunctor $F:=F_{\proj\,(\Lambda)}$ coincides with 
$\Ext^1_\Lambda(-,-).$ Therefore, the $\X$-relative Igusa-Todorov functions are just the canonical ones. 
Hence, the result follows from   Corollary  \ref{experp2}, since $\id_F(\proj\,(\Lambda))=\id\,({}_\Lambda\Lambda).$
\end{proof}

\begin{remark} Let $\Lambda$ be an artin algebra. The inequality $\mathrm{findim}(\Lambda)\leq  \id\,({}_\Lambda\Lambda),$ was already known \cite{Bass} by H. Bass in 1962. The inequality $\Psi\mathrm{dim}\,(\Lambda)\leq \id\,({}_\Lambda\Lambda)$ appears in \cite[Corollary 3.7]{LM}.
\end{remark}

\section{Homologically compatible pairs and balance} 
In this section we assume that $\C$ is an abelian category. Moreover, we use some 
results and the notation given in  \cite{Holm}. The material presented here deals with the right derived functors 
$\Ext^i_\X(-,-)$ and $\Ext^i_\Y(-,-)$ of $\Hom(-,-),$ for some ``good" pair of classes $\X$ and $\Y.$ 
We discuss the balance of these functors and give  applications in Gorenstein homological algebra via the 
 Igusa-Todorov functions.

Let $\X$ be a class of objects in $\C.$ We denote by $\LR(\X)$ the class of all objects $M$ admitting 
a {\bf proper $\X$-resolution} $X_\bullet(M)\stackrel{\varphi}{\to} M,$ that is, a complex
$$\eta_M:\quad\cdots\to X_1\to X_0\stackrel{\varphi}{\to} M\to 0$$
 such that $X_i\in\X$ and the Hom-complex $\Hom(X,\eta_M)$ is acyclic for any $X\in\X.$ Using such a complex 
 $X_\bullet(M)\stackrel{\varphi}{\to} M,$ 
we can define, in the usual way, the derived functor $\Ext^i_\X(M,N)$ for any $N\in\C.$ Dually, 
for any $\Y\subseteq\C,$ we have the 
 class $\RR(\Y)$ of all objects $N$ admitting a {\bf proper $\Y$-coresolution} 
$N\to Y_\bullet(N).$ Thus, we also get the derived functor $\Ext^i_\Y(M,N)$ for 
any $M\in\C.$

\begin{defi} A pair $(\X,\omega)$ of classes of objects in an abelian category $\C$ is {\bf left compatible} if $\X$ is closed under 
extensions and direct summands, and $\omega$ is an $\X$-injective relative cogenerator in $\X.$ 
Dually, we also have the notion of {\bf right compatible} pair $(\varepsilon,\Y),$ where $\varepsilon$ is a $\Y$-projective relative 
generator in $\Y.$
\end{defi}

\begin{remark}\label{LC1} Let $(\X,\omega)$ be a left compatible pair in an abelian category $\C.$ Consider any $M\in\X^\wedge.$ Then, by 
Lemma \ref{tresdim} (a), there is an exact sequence $0\to K\to X\stackrel{\varphi}{\to}M\to 0,$ with 
$\resdim_\omega(K)=\resdim_\X(M)-1$ and $\varphi:X\to M$ an $\X$-precover. So, we have that 
$\X^\wedge\subseteq\LR(\X).$
\end{remark}

\begin{pro}\label{ProBal1} Let $(\X,\omega)$ be a left compatible pair in an abelian category $\C.$ If $\X^\wedge=\C$ then the following statements 
hold true, for $F:=F_\X.$
\begin{itemize}
\item[(a)] The pair $(\C,\E_F)$ is an exact category with enough projectives and $\Q(\E_F)=\X.$
\item[(b)] $\pd_F(M)=\resdim_\X(M)$ for any $M\in\C.$
\item[(c)] $\pd_F(\C)=\id\,(\omega).$
\item[(d)] Let $\C$ be skeletally small and Krull-Schmidt. Then 
 \begin{itemize}
 \item[(d1)] $\fpd_F(\C)\leq\Fixdim\,(\C)\leq\Psixdim\,(\C)\leq\id_F(\X),$
 \item[(d2)] $\id\,(\omega)<\infty$ $\Rightarrow$ $\Fixdim\,(\C)=\Psixdim\,(\C)=\id\,(\omega).$
 \end{itemize}
\end{itemize}
\end{pro}
\begin{dem} (a)  Using Lema \ref{tresdim} (a), we can proceed as in the proof of Proposition \ref{projExcat}.
\

(b) Let $M\in\C.$ Then,  from item (a) and Lemma \ref{primeraI}, we get that 
$$\pd_F(M)=\Omega_\X\mathrm{dim}(M)=\mathrm{t.resdim}_\X(M).$$ 
Finally, by Lemma \ref{tresdim} (b), we conclude that $\pd_F(M)=\resdim_\X(M).$
\

(c) It follows from (b), \cite[Proposition 2.1]{ABu} and Lemma \ref{Lcam}.
\

(d1) By the item (a) and the fact that $\C$ is abelian give us that $(\C,\E_F)$ is an exact $IT$-context, which is weakly idempotent complete. Thus, the item (d1) follows from Theorem \ref{experp1.5}.
\

(d2) It follows from (c) and Remark \ref{primeras} (a), since we know that $(\C,\E_F)$ is an exact $IT$-context, which is weakly idempotent complete.
\end{dem}

The following result, given by H. Holm in \cite{Holm}, is a nice criteria to get balance for 
bifunctors. We only use the version that gives us balance for the relative $\Ext$ bi-functors.

\begin{teo}\cite[Theorem 2.6]{Holm}\label{HolmBalance} 
Let $\X\subseteq\tilde{\X}$ and $\Y\subseteq\tilde{\Y}$ be classes of 
objects in an abelian category $\C$ satisfying the following conditions.
\begin{itemize}
\item[(a)] Any $M\in\tilde{\X}$ admits a proper $\X$-resolution 
$\eta_M:\quad X_\bullet(M)\to M$ such that the complex $\Hom(\eta_M,Y)$ is acyclic 
for any $Y\in\Y.$
\item[(b)] Any $N\in\tilde{\Y}$ admits a proper $\Y$-coresolution 
$\theta_N:\quad N\to Y_\bullet(N)$ such that the complex $\Hom(X,\theta_N)$ is acyclic 
for any $X\in\X.$
\end{itemize}
Then, for any $i\geq 1,$ we have the following functorial isomorphism
$$\Ext^i_\X(M,N)\simeq\Ext^i_\Y(M,N)$$ 
for all $M\in\tilde{\X}$ and $N\in\tilde{\Y}.$
\end{teo} 

In what follows, we apply Holm's criteria and Holm's techniques \cite{Holm} in two situations. The first one is the following.

\begin{defi} The {\bf pairs} $(\X,\omega)$ and $(\varepsilon,\Y)$ in an abelian category $\C$ are {\bf homologically compatible} if $(\X,\omega)$ is left compatible, $(\varepsilon,\Y)$ is right compatible 
and the balanced conditions hold:
\begin{itemize}
\item[(B1)] $\Ext^1_\C(\X^\wedge,\varepsilon)=0=\Ext^1_\C(\omega^\wedge,\Y),$
\item[(B2)] $\Ext^1_\C(\X,\varepsilon^\vee)=0=\Ext^1_\C(\omega,\Y^\vee).$
\end{itemize}
\end{defi}

\begin{pro}\label{balance1} Let $(\X,\omega)$ and $(\varepsilon,\Y)$ be homologically compatible pairs in an abelian category $\C.$ Then, for any $i\in\N,$ $M\in\X^\wedge$ and $N\in\Y^\vee,$ there is an 
isomorphism
$$\Ext^i_\X(M,N)\simeq\Ext^i_\Y(M,N),$$
which is functorial in both variables.
\end{pro}
\begin{dem} We only check condition (a) in Theorem \ref{HolmBalance}, since condition (b) follows in a similiar way. Indeed, by Remark \ref{LC1} we have that $\X^\wedge\subseteq\LR(\X).$ Moreover, for 
$M\in\X^\wedge,$ there is an exact sequence 
$$\theta:\quad 0\to K\stackrel{\varphi'}{\to} X\stackrel{\varphi}{\to}M\to 0,$$ 
with $K\in\omega^\wedge$ and $\varphi:X\to M$ an $\X$-precover. It is enough to see that 
any morphism $\alpha:K\to Y,$ with $Y\in\Y,$ factors through $\varphi':K\to X.$ Using that $\varepsilon$ 
is a relative generator in $\Y,$ we get an exact sequence
$$\eta:\quad 0\to Y'\to E\stackrel{\lambda}{\to} Y\to 0,$$
with $Y'\in\Y$ and $E\in\varepsilon.$ Applying the functor $\Hom(K,-)$ to $\eta$ and using that 
$\Ext^1_\C(\omega^\wedge,\Y)=0,$ it follows the existence of $\alpha':K\to E$ such that $\alpha=\lambda\alpha'.$ Now, we apply the funtor $\Hom(-,E)$ to $\theta.$ Hence, using that $\Ext^1_\C(\X^\wedge,\varepsilon)=0,$ there is $\alpha'':X\to E$ such that $\alpha'=\alpha''\varphi'.$ Therefore
$\lambda\alpha''\varphi'=\lambda\alpha'=\alpha,$ proving that $\alpha$ factors trough $\varphi'.$
\end{dem}

\begin{teo}\label{Sbalance1} Let $(\X,\omega)$ and $(\varepsilon,\Y)$ be homologically compatible pairs in an abelian category $\C$  such that 
$\X^\wedge=\C=\Y^\vee.$ Then, for the additive subfunctors $F:=F_\X$ and 
$G:=F^\Y,$ the following statements hold.
\begin{itemize}
\item[(a)] For any $M,N\in\C$ there is an 
isomorphism
$$\Ext^i_\X(M,N)\simeq\Ext^i_\Y(M,N),$$
which is functorial in both variables.
\item[(b)] $\pd_G(M)=\pd_F(M)=\resdim_\X(M)$ for any $M\in\C.$
\item[(c)] $\id_F(N)=\id_G(N)=\coresdim_\Y(M)$ for any $N\in\C.$
\item[(d)] $\pd\,(\varepsilon)=\id_G(\C)=\pd_F(\C)=\id\,(\omega).$
\item[(e)] Let $\C$ be skeletally small and Krull-Schmidt. Then, we have the following:  
  \begin{itemize}
  \item[(e1)] $\fpd_F(\C)\leq\Fixdim\,(\C)\leq\Psixdim\,(\C)\leq\coresdim_\Y(\X),$
  \item[(e2)] $\finid_G(\C)\leq\Phi^\Y\mathrm{dim}\,(\C)\leq\Psi^\Y\mathrm{dim}\,(\C)\leq\resdim_\X(\Y),$
  \item[(e3)] in case $\id\,(\omega)$ is finite, we get the equalities
  $$\Fixdim\,(\C)=\Psixdim\,(\C)=\coresdim_\Y(\X)=\resdim_\X(\Y)=\Phi^\Y\mathrm{dim}\,(\C)=$$ $$=\Psi^\Y\mathrm{dim}\,(\C)=\id\,(\omega). $$
   
 \end{itemize}
\end{itemize}
\end{teo}
\begin{dem} (a) It follows from Proposition \ref{balance1}. 
\

(b) It follows from (a) and Proposition \ref{ProBal1} (b).
\

(c) The proof of (c)  is dual to the given in (b).
\

(d)  The equalities $\pd\,(\varepsilon)=\id_G(\C)$ and $\pd_F(\C)=\id\,(\omega)$ can be obtained from Proposition \ref{ProBal1} (c) and its dual. Finally, the equality $\id_G(\C)=\pd_F(\C)$ follow from (a).
\

(e1),  (e2)  They be obtained from (b), (c), Proposition \ref{ProBal1} (c) and its dual.
\

(e3) It follows from (d),  (e1) and (e2).
\end{dem}

\begin{pro}\label{ProSbalance2} Let $(\X,\Y)$ be a complete hereditary cotorsion pair in an abelian category $\C,$ with 
enough projectives. Then, the following statements hold true, for $F:=F_\X$ and $\omega:=\X\cap\Y.$
\begin{itemize}
\item[(a)] The pair $(\X,\omega)$ is left compatible.
\item[(b)] The pair $(\C,\E_F)$ is an exact category with enough projectives and $\Q(\E_F)=\X.$
\item[(c)] $\pd_F(M)=\resdim_\X(M)$ for any $M\in\C.$
\item[(d)] $\fpd_F(\C)=\pd_F\,(\X^\wedge)=\id_{\X^\wedge}\,(\omega)\leq\id\,(\omega).$
\item[(e)] Let $\C$ be skeletally small and Krull-Schmidt. Then, we have that 
 \begin{itemize}
 \item[(e1)] $\fpd_F(\C)\leq\Fixdim\,(\C)\leq\Psixdim\,(\C)\leq\id_F(\X),$
 \item[(e2)] $\X^\wedge=\C$ and $\id\,(\omega)<\infty$ $\Rightarrow$ $\Fixdim\,(\C)=\Psixdim\,(\C)=\id\,(\omega).$
 \end{itemize}
\end{itemize}
\end{pro}
\begin{dem} (a)  By the definition of complete cotorsion pairs, it can be seen that 
$\omega$ is a relative cogenerator in $\X.$ Furthermore,  from the fact that $\Ext^i_\C(X,Y)=0$ for any 
$i\geq 1,$ $X\in\X$ and $Y\in\Y,$ we get that $(\X,\omega)$ is left compatible.
\

(b) It follows from Proposition \ref{projExcat}, since $\X$ is a precovering generator in $\C.$
\

(c), (d) They follows from Proposition \ref{igualdades}, since $\X$ is a generator precovering class in $\C.$ 
\

(e1) By (b) we get that the pair $(\C,\E_F)$ is an exact $IT$-context which is weakly idempotent complete. Hence, the item (e1) follows from Theorem \ref{experp1.5}.
\

(e2) Since $\X^\wedge=\C$ and $\id\,(\omega)<\infty,$ it follows from (d) that $\pd_F(\C)=\id\,(\omega)<\infty.$ Thus, the equalities in (e2) follow from Remark \ref{primeras} (a). 
\end{dem}

\begin{cor}\label{CoroSbalance2} Let $(\X,\Y)$ be a complete hereditary cotorsion pair in an abelian skeletally small and Krull-Schmidt category $\C,$ with 
enough projectives and injectives. Then, for $\omega:=\X\cap\Y,$ the following statements hold true.
\begin{itemize}
\item[(a)] $\X^\wedge=\C$ and $\id\,(\omega)<\infty$ $\quad\Rightarrow\quad$ $\Fixdim\,(\C)=\Psixdim\,(\C)=\id\,(\omega).$
\item[(b)] $\Y^\vee=\C$ and $\pd\,(\omega)<\infty$ $\quad\Rightarrow\quad$ $\Phi^{\Y}\mathrm{dim}\,(\C)=\Psi^{\Y}\mathrm{dim}=\pd\,(\omega).$
\end{itemize}
\end{cor}
\begin{dem} It follows from Proposition \ref{ProSbalance2} and its dual. 
\end{dem}

The second situation where we can apply Holm's criterion is for some special kind of cotorsion pairs. To start with, we introduce the following notion.

\begin{defi} Two {\bf complete cotorsion pairs} $(\X,\mathcal{Z})$ and $(\mathcal{Z},\Y)$ in an abelian category $\C$ are {\bf homologically compatible} if 
 $\omega:=\X\cap\mathcal{Z}=\Q(\C)$ and $\varepsilon:=\mathcal{Z}\cap\Y=\I(\C).$ 
\end{defi}

\begin{lem}\label{LSbalance2} Let $(\X,\mathcal{Z})$ and $(\mathcal{Z},\Y)$ be homologically 
compatible complete cotorsion pairs in $\C.$ Then, for any $M,N\in\C$ there is an 
isomorphism
$$\Ext^i_\X(M,N)\simeq\Ext^i_\Y(M,N),$$
which is functorial in both variables.
\end{lem}
\begin{dem} By the definition of complete cotorsion pairs, it can be seen that 
$\omega:=\X\cap\mathcal{Z}$ is a relative cogenerator in $\X;$ and furthermore, the class 
$\varepsilon:=\mathcal{Z}\cap\Y$ is a relative generator in $\Y.$ Thus, the same proof given 
in Proposition \ref{balance1} works here, but now using different reasons (the properties of homologically compatible cotorsion pairs). 
\end{dem}

\begin{teo}\label{Sbalance2} Let $(\X,\mathcal{Z})$ and $(\mathcal{Z},\Y)$ be homologically 
compatible complete hereditary cotorsion pairs in an abelian category $\C,$ with enough projectives and injectives. Then, for the additive subfunctors $F:=F_\X$ and 
$G:=F^\Y,$ the following statements hold true.
\begin{itemize}
\item[(a)] $\pd_G(M)=\pd_F(M)=\resdim_\X(M)$ for any $M\in\C.$
\item[(b)] $\id_F(N)=\id_G(N)=\coresdim_\Y(M)$ for any $N\in\C.$
\item[(c)] $\pd_{\Y^\vee}\,(\I(\C))=\fid_G\,(\C)=\fpd_F\,(\C)=\id_{\X^\wedge}\,(\Q(\C)).$
\item[(d)] Let $\C$ be skeletally small and Krull-Schmidt. Then
  \begin{itemize}
  \item[(d1)] $\fpd_F(\C)\leq\Fixdim\,(\C)\leq\Psixdim\,(\C)\leq\coresdim_\Y(\X),$ 
  \item[(d2)] $\finid_G(\C)\leq\Phi^\Y\mathrm{dim}\,(\C)\leq\Psi^\Y\mathrm{dim}\,(\C)\leq\resdim_\X(\Y),$
  \item[(d3)] $\Y\subseteq\X^\wedge$ $\quad\Rightarrow\quad$ $\finid_G(\C)=\Phi^\Y\mathrm{dim}\,(\C)=\Psi^\Y\mathrm{dim}\,(\C)=\resdim_\X(\Y),$
   \item[(d4)] $\X\subseteq\Y^\vee$ $\quad\Rightarrow\quad$ $\fpd_F(\C)=\Fixdim\,(\C)=\Psixdim\,(\C)=\coresdim_\Y(\X).$
 \end{itemize}
\end{itemize}
\end{teo}
\begin{dem}  (a)  It follows from Lemma \ref{LSbalance2} and Proposition \ref{ProSbalance2} (c).
\

(b) It is dual to the proof of (a).
\

(c) The equality $\fid_G\,(\C)=\fpd_F\,(\C)$ is a consequence of Lemma \ref{LSbalance2}. To complete the proof, we use Proposition 
\ref{ProSbalance2} (d) and its dual.
\

(d) Let $\C$ be skeletally small and Krull-Schmidt. Then, by (b) and Proposition \ref{ProSbalance2} (e1), we get (d1). Finally, the proof of (d2) is dual. 
\

The proof of (d3) can be done as follows. Firstly, the item (a) say us that $\fpd_F(\C)=\resdim_\X(\X^\wedge)$ since 
$\Q^{<\infty}_F(\C)=\X^\wedge.$ Thus, from (c) it follows that $\fid_G(\C)=\resdim_\X(\X^\wedge).$ If we assume that $\Y\subseteq\X^\wedge,$ so we get that $\resdim_\X(\Y)\leq \resdim_\X(\X^\wedge)=\fid_G(\C).$ Thus, we conclude (d3) from (d2). Finally, note that the proof of (d4) is similar to the one given in  (d3).
\end{dem}

\begin{remark}\label{Rkbalance2} Let $\C$ be an abelian category with enough projectives and injectives.
\begin{itemize}
\item[(a)]\cite{GR} For a cotorsion pair $(\X,\Y),$  we have that: 
$(\X,\Y)$ is hereditary $\Leftrightarrow$ $\X$ is resolving $\Leftrightarrow$ $\Y$ is coresolving.
\item[(b)]\cite{BR} Let $\Q(\C)$ and $\I(\C)$ be functorially finite in $\C.$ Then, by  \cite[Theorem 3.2]{BR} and its proof there, we get the following statements. 
 \begin{itemize}
 \item[(b1)] If $(\X,\mathcal{Z})$ and $(\mathcal{Z},\Y)$ are complete cotorsion pairs and 
 $\mathcal{Z}$ is resolving and coresolving, then we have that $\mathcal{Z}$ is functorially finite, 
 $\X\cap\mathcal{Z}=\Q(\C)$ and $\mathcal{Z}\cap\Y=\I(\C).$ 
 \item[(b2)] Let $\C$ be a KS-category and $\mathcal{Z}$ be functorially finite, resolving and coresolving.  Then, the pairs  
 $({}^{\perp_1}\mathcal{Z},\mathcal{Z})$ and $(\mathcal{Z},\mathcal{Z}^{\perp_1})$ are hereditary complete cotorsion pairs. Moreover, by (b1), we get that these pairs are homologically compatible.
 \end{itemize}
 \item[(c)] In \cite[Theorem 4.5]{BR} there is a proof of Lemma \ref{LSbalance2}  under stronger 
asumptions. The authors assume that $\Q(\C)$ and $\I(\C)$ are functorially finite in $\C$ 
and take hereditary cotorsion pairs (see (b) above). Furthermore they use the machinary of Tate-Vogel 
cohomological bifunctors to get the result. Our proof is rather elementary and it is not necessary to assume the conditions given in (b1).
\end{itemize}
\end{remark}

\section{Relative $n$-Igusa-Todorov categories}

In this section, we generalise the notion of $n$-Igusa-Todorov algebras ($n$-IT algebras, for short), which were introduced by J. Wei 
in \cite{Wei}. This generalisation will be made into the setting of strong $ IT$-contexts. 
\

We recall, that for 
a strong abelian $ IT$-context  $\C$ and a class $\X\subseteq \C,$ which is a precovering 
generator class in $\C,$ the class $\X$ induces a  functor $F:=F_\X.$ In such a case the pair $(\C,F)$ is an strong exact $ IT$-context. Furthermore, if $\X'$ is a precovering generator class in $\C,$ such that $\X\subseteq\X',$ then $(\C,F_{\X'},F_{\X})$ is a strong triple $ IT$-context (see Definition \ref{DefIT-triple}).
\

Let $(\C,\E)$ be an $ IT$-context. We say that a given class $\mathcal{Z},$ of objects in $\C,$ is of finite  type if there is some object $M\in\C$ such that 
$\add\,(\mathcal{Z})=\add\,(M).$ In such a case, we will say that $\mathcal{Z}$ is of finite type $M.$ Note that, if the class 
$\mathcal{Z}$ is of finite type $M$, then by  Proposition \ref{propsi} (b)  we get that  $\Psiedim\,(\mathcal{Z})=\Psi_\E\,(M)<\infty.$ 

\begin{defi} An $n$-IT category is a strong triple $ IT$-context $(\C,\E',\E)$ satisfying the following two 
conditions:
\begin{itemize}
\item[(a)] for some $n\in\N$ and any $M\in\C$ there is an $\E$-exact sequence 
$$ V_1\rightarrowtail V_0\twoheadrightarrow \Omega_\E^n(M)\oplus P,$$ 
with $V_1,V_0\in\Q(\E')$ and $P\in\Q(\E);$
\item[(b)] the dimension $\Psiedim\,(\Q(\E'))$ is finite.
\end{itemize}
\end{defi}

\begin{remark} Let $\Lambda$ be an artin algebra, which is an $n$-IT algebra and has $n$-IT module $V$ in the sense of 
\cite[Definition 2.2]{Wei}. Consider the length-category $\C:=\modu\,(\Lambda)$ and the classes 
 $\X:=\proj\,(\Lambda)$ and $\X':=\add\,(V\oplus{}_\Lambda\Lambda).$ Then, the triple 
 $(\C,F_{\X'},F_{\X})$ is an $n$-IT category. 
\end{remark}

The result \cite[Theorem 2.3]{Wei} can be generalised for $n$-IT categories as follows. 

\begin{teo}\label{ITC1} If $(\C,\E',\E)$ is an $n$-IT category, then 
$$\findim_\E(\C)\leq n+1+\Psiedim\,(\Q(\E'))<\infty.$$
\end{teo} 
\begin{dem} Assume that $(\C,\E',\E)$ is an $n$-IT category. Then, for any $M\in \mathcal{P}^{<\infty}_{\E}(\C),$ there 
is an $\E$-exact sequence $V_1\rightarrowtail V_0\twoheadrightarrow \Omega_\E^n(M)\oplus P,$ with $V_1,V_0\in\Q(\E')$ and $P\in\Q(\E).$ Hence, by 
Theorem \ref{desigualdad} and Lemma \ref{primeraI}, we get 
$$\pd_\E(M)\leq n+\pd_\E(\Omega_\E^n(M))\leq n+1+\Psi_\E\,(V_1\oplus V_0);$$
proving the result, since $\Psi_\E\,(V_1\oplus V_0)\leq \Psiedim\,(\Q(\E')).$
\end{dem}

\begin{pro}\label{ITC2} Let $(\C,\E',\E)$ be a strong triple $ IT$-context such that $\Psiedim\,(\Q(\E'))$ is 
finite. Let us consider the following conditions, for some $n\in\N:$ 
\begin{itemize}
\item[(a)]   $\pd_{\E'}(\Omega^{n}_\E(\C))\leq 1,$
\item[(b)] $\pd_{\E'}(\C)\leq n+1$  and  $\Omega_{\E}^n(M)|\Omega_{\E'}^n(M)$ for any $M\in\C.$
\end{itemize}
If one of the two conditions {\rm (a)} and {\rm (b)} holds true, then $(\C,\E',\E)$ is an $n$-IT category.
\end{pro}
\begin{dem} Note that condition (b) implies (a). Assume that the conditions (a)  hold true. Then, for any $M\in\C,$ we have 
an $\E'$-exact sequence 
$$\eta:\quad V_1\rightarrowtail V_0\twoheadrightarrow \Omega_{\E}^n(M),$$
 with $V_0, V_1 \in\Q(\E').$ But, we know $\E'\subseteq\E$  and hence $\eta$ is also an $\E$-exact sequence. Therefore,  we get that $\eta$ satisfies the desired conditions.
\end{dem}

\begin{cor}\label{ITC3}  If a strong triple $ IT$-context $(\C,\E',\E)$ satisfies that $\pd_{\E'}(\C)\leq 1$ and $\Psiedim\,(\Q(\E'))$ is finite,     then 
$(\C,\E',\E)$ is a $0$-IT category.
\end{cor}
\begin{dem} For the case $n=0,$ the condition (b) in Proposition \ref{ITC2} holds easily, since $\Omega_{\E'}^0(M)=M=
\Omega_{\E}^0(M).$
\end{dem}

In what follows, we construct an example of a $k$-$ IT$-category for some non-negative integer $k.$  Let $\Lambda$ be an artin algebra, and let $M$ be an Auslander generator in $\modu\,(\Lambda).$ Assume that  $\Lambda$  is not 
of finite representation type. We recall that, for  an Auslander generator $M\in\modu\,(\Lambda),$ it follows that 
$\gldim_{G}(\Lambda)=\repdim\,(\Lambda)-2$ is finite,  where $G:=F_{\add\,(M)}.$ Therefore, the positive integer 
$$t:=\min\{n\in\N\;:\;\Omega^n_G\,(\modu\,(\Lambda)) \text{ is of finite type}\}$$ is well defined. Note that the class 
$\X:=\add\,(M)\oplus\Omega^t_G\,(\modu\,(\Lambda))$ is of finite type, i. e. there exist $N\in\modu\,(\Lambda)$ such that 
$\add(\,\X)=\add\,(N).$ Therefore, by Proposition \ref{propsi} it follows that 
$\Psi_{\add\,(M)}\mathrm{dim}\,(\X)=\Psi_{\add\,(M)}\,(N)<\infty.$ 
\

Consider the subfunctor  $F:=F_\X.$ We assert that the strong triple $ IT$-context $(\modu\,(\Lambda),\E_F,\E_G)$ is an $(t-1)-IT$ category. 
Indeed, let $X\in\modu\,(\Lambda).$ Then, we have a $G$-exact sequence 
$\eta\;0\to\Omega^t_G(X)\to M_0\to \Omega^{t-1}_G(X)\to 0,$ for some $M_0\in\add\,(M).$ Note that $\eta$ is the desired exact sequence since $\Q(\E_F)=\X.$

\section{Examples}
Throughout this section $\Lambda$  denotes an  artin $R$-algebra, where $R$ is a commutative artinian ring. The category of finitely
generated left $\Lambda$-modules is denoted by $\modu\,(\Lambda)$ and the full
subcategory of finitely generated projective $\Lambda$-modules by
$\proj\,(\Lambda).$ Similarly, $\inj\,(\Lambda)$ stands for the subcategory of finitely generated injective $\Lambda$-modules. 
Let $E$ be the injective envelope of  the quotient $R$-module $R/\rad(R).$  We recall 
that the functor $D:=\Hom_R(-,E):\modu\,(\Lambda)\to\modu\,(\Lambda^{op})$ is a duality of categories. Moreover, $\modu\,(\Lambda)$ is an abelian category with enough projectives and injectives. In what follows, we give applications of the developed theory.

\subsection{Tilting and Cotilting}
\

We use a concrete cotorsion pair induced by a finitely generated cotilting $\Lambda$-module $C,$ where $\Lambda$ is an artin algebra \cite{Mi}. Recall that $C$ is {\bf cotilting} if $\id\,(C)$ is finite, $\Ext^i_{\Lambda} (C,C)=0$ for any 
$i\geq 1,$ and $\inj\,(\Lambda)\subseteq (\add\,(C))^\wedge.$ 

\begin{teo} For a cotilting $\Lambda$-module $C\in\modu\,(\Lambda)$ and the subfunctor $F:=F_{{}^\perp C},$ the following statements hold true.
\begin{itemize}
\item[(a)] $\fpd\,(\Lambda)\leq 1+s+\Psi\mathrm{dim}\,({}^\perp C\oplus\Omega_F\Omega^s(\Q^{<\infty}(\Lambda)))$ for any $s\in\N,$
\item[(b)] $\Phi_{{}^\perp C}\mathrm{dim}\,(\Lambda)=\Psi_{{}^\perp C}\mathrm{dim}\,(\Lambda)=\id\,(C).$
\end{itemize}
\end{teo}
\begin{dem} By \cite{Re} we get that $({}^\perp C, ({}^\perp C)^\perp)$ is a complete hereditary cotorsion pair such that 
$({}^\perp C)^\wedge=\modu\,(\Lambda).$ Therefore, the result follows from Theorem \ref{clasico1} and Corollary \ref{CoroSbalance2} (a).
\end{dem}

The dual of the above result also holds, and for the convenience of the reader, we state it in the following remark. For doing so, we recall   
the notion of tilting module. It is said that $T$ is a {\bf tilting} $\Lambda$-module 
if $\pd\,(T)$ is finite, $\Ext^i_{\Lambda} (T,T)=0$ for any $i\geq 1,$ and $\proj\,(\Lambda)\subseteq (\add\,(T))^\vee.$

\begin{remark} For a tilting $\Lambda$-module $T\in\modu\,(\Lambda)$ and the subfunctor $G:=F^{T^\perp },$  the following statements hold true. 
\begin{itemize}
\item[(a)] $\fid\,(\Lambda)\leq 1+s+\Psi^{\inj(\Lambda)}\mathrm{dim}\,(T^\perp \oplus\Omega^{-1}_G\Omega^{-s}(\I^{<\infty}(\Lambda)))$ for any $s\in\N,$
\item[(b)] $\Phi^{T^\perp}\mathrm{dim}\,(\Lambda)=\Psi^{T^\perp}\mathrm{dim}\,(\Lambda)=\pd\,(T).$
\end{itemize}
\end{remark}

\subsection{Standarly stratified algebras.}  For $\Lambda$-modules $M$ and $N$,
$\Tr_M\,(N)$ is the trace of $M$ in $N$, that is, $\Tr_M\,(N)$ is
the $\Lambda$-submodule of $N$ generated by the images of all
morphisms from $M$ to $N$.
\

Let $n$ be the rank of the Grothendieck group $K_0\,(\Lambda)$. We fix a linear order $\leq$ on the set $[1,n]:=\{1,2,\cdots,n\}$ and a representative set 
$ P=\{P(i)\;:\;i\in[1,n]\}$ containing one module of each iso-class in $\proj\,(\Lambda).$ The set of standard $\Lambda$-modules is $\Delta=\{
\Delta(i):i\in[1,n]\},$ where
$\Delta(i)= P(i)/\Tr_{\oplus_{j>i}\,P(j)}\,(P(i))$. 
\

Let $\F(\Delta)$ be the subcategory of $\modu\,(\Lambda)$ consisting of the
$\Lambda$-modules having a $\Delta$-filtration, that is,
a filtration  $0=M_0\subseteq M_1\subseteq\cdots\subseteq M_s=M$ with factors $M_{i+1}/M_i$
isomorphic to a module in $\Delta$ for all $i$. The algebra $\Lambda$ is
a standardly stratified algebra, with respect to the linear
order $\leq$ on the set $[1,n]$, if
$\proj\,(\Lambda)\subseteq\F(\Delta)$ (see \cite{ADL,Dlab,CPS}). A standardly stratified algebra $\Lambda$ is called 
quasi-hereditary if $\End_\Lambda\,(\Delta(i))$ is a division ring for any $i\in[1,n].$ We also consider the class of 
$\Delta$-injective modules $\I(\Delta):=\F(\Delta)^{\perp_1}.$
\

It is well known, see in  \cite[Theorem 1.6, Proposition 2.2]{AHLU},   that
 $(\F(\Delta),\I(\Delta))$ is a complete hereditary cotorsion pair in $\modu\,(\Lambda).$ Furthermore, there exists 
 a basic tilting module $T$ (known as the characteristic tilting module) such that $\F(\Delta)\cap\I(\Delta)=\add\,(T).$

\begin{teo}\label{experp2.3} For a  standarly stratified algebra $\Lambda,$ the following statements hold true, where 
$F:=F_{\F(\Delta)}$ and $T$ is the characteristic tilting module.
\begin{itemize}
\item[(a)] $\resdim_{\F(\Delta)}(\Q^{<\infty}_F(\Lambda))\leq\Phi_{\F(\Delta)}\mathrm{dim}\,(\Lambda)\leq\Psi_{\F(\Delta)}\mathrm{dim}\,(\Lambda)\leq\id\,(\Lambda).$
\item[(b)] If $\Lambda$ is a quasi-hereditary, then
 \begin{itemize}
 \item[(b1)] 
 $\Phi_{\F(\Delta)}\mathrm{dim}\,(\Lambda)=\Psi_{\F(\Delta)}\mathrm{dim}\,(\Lambda)=\id\,(T).$
 \item[(b2)] 
 $\Phi^{\I(\Delta)}\mathrm{dim}\,(\Lambda)=\Psi^{\I(\Delta)}\mathrm{dim}\,(\Lambda)=\pd\,(T).$
 \end{itemize}
\end{itemize}
  
\end{teo}
\begin{proof}  (a) Since $\Lambda$ is a  standardly stratified algebra, we get from  \cite[Corollary 2.5 (b)]{MMS3}  that 
$\id\,(\F(\Delta))=\id\,({}_\Lambda\Lambda).$ Thus, (a)   follows from Corollary \ref{experp2.2} (a).
\

(b) Let $\Lambda$ be quasi-hereditary. It  is well known that, in this case, $\Lambda$ has finite global dimension. Therefore 
$\F(\Delta)^\wedge=\modu\,(\Lambda)=\I(\Delta)^\vee.$ Using the fact that $(\F(\Delta),\I(\Delta))$ is a complete 
hereditary cotorsion pair in $\modu\,(\Lambda)$ such that 
$\F(\Delta)\cap\I(\Delta)=\add\,(T),$ the result follows from Corollary \ref{CoroSbalance2}.
\end{proof}

\subsection{Gorenstein homological algebra in $\modu\,(\Lambda)$}
\

In order to give  applications  of the developed theory, using Gorenstein homological algebra,  we need to recall the notion of Gorenstein projective (respectively, injective) 
objects in abelian categories. Let $\C$ be an abelian category with enough projectives and injectives. An acyclic complex 
$$\eta_X:\quad \cdots\to X^i\to X^{i+1}\to\cdots$$ of projective (respectively, injective) objects in $\C$ is called totally acyclic if the complex 
$\Hom_\C(\eta,X)$ (respectively, $\Hom_\C(X,\eta,)$ is acyclic for any projective (respectively, injective) object $X\in\C.$ The classes of Gorenstein projective and Gorenstein injective objects were introduced by Enoch and Jenda in \cite{EJ2}, building on the notion of a finitely generated module 
of $G$-dimension zero introduced by Auslander in \cite{A1}. A Gorenstein projective object $C\in\C$ is of the form $C=\Coker\,(X^{-1}\to X^0)$ for some totally acyclic complex $\eta_X$ of projective objects. Dually, a Gorenstein injective object $C\in\C$ is of the form $C=\Ker\,(X^0\to X^1)$ for some totally acyclic complex $\eta_X$ of injective objects. We denote by $\GP(\C)$ the class of Gorenstein projective objects and by $\GI(\C)$ the class of Gorenstein injective objects in the abelian category $\C.$

Let $\Lambda$ be an artin algebra. We consider the  category $\Mod\,(\Lambda)$ of all left 
$\Lambda$-modules and the full subcategory $\modu\,(\Lambda)$ of all finitely generated left 
$\Lambda$-modules. 

By taking $\C:=\Mod\,(\Lambda),$ we get the  Projective $\Lambda$-modules 
 $\Proj(\Lambda):=\Q(\C)$ and the class $\GProj(\Lambda):=\GP(\C)$ of  Gorenstein Projective $\Lambda$-modules. Furthermore, for any $M\in\Mod\,(\Lambda),$ we have the  Gorenstein Projective 
 dimension $\GPd(M):=\resdim_{\GP(\C)}(M)$ and  $\GProj^{<\infty}(\Lambda):=\{M\in\C\;:\;
 \GPd(M)<\infty\}.$
 
Similarly, replacing $\Mod\,(\Lambda)$ by $\modu\,(\Lambda),$ we get the  finitely generated projective $\Lambda$-modules $\proj(\Lambda),$  the class $\Gproj(\Lambda)$ of the small Gorenstein projective $\Lambda$-modules, the small Gorenstein projective 
 dimension $\Gpd(M):=\resdim_{\Gproj(\Lambda)}(M),$ for any $M\in\modu\,(\Lambda),$ and the class 
 $\Gproj^{<\infty}(\Lambda):=\{M\in\C\;:\;
 \Gpd(M)<\infty\}.$ Dually, we can define the similar notions  for injectives and Gorenstein injectives.

\begin{remark}\label{basicosG} For an artin algebra $\Lambda,$ we have the following properties.
\begin{itemize}
\item[(a)] By \cite[Proposition 1.4]{PZ}, we get 
$$\GProj(\Lambda)\cap\modu\,(\Lambda)=\Gproj(\Lambda).$$
\item[(b)] Using (a) and \cite[Remark 2.12]{Holm2}, it follows
$$\GPd(M)=\Gpd(M)\quad\;\text{for any}\;M\in\modu\,(\Lambda).$$
\item[(c)] $\GProj^{<\infty}(\Lambda)\cap\modu\,(\Lambda)=\Gproj^{<\infty}(\Lambda).$
\item[(d)] Let $\C:=\modu\,(\Lambda)$  and the exact structure $\E=\E_{\max}$ of all exact sequences in $\C.$ Then, by Corollary \ref{experp1} and the fact that $\Gproj(\Lambda)\subseteq{}^\perp\Lambda,$ we get (see also in \cite{LM})
$$\Phidim(\Gproj(\Lambda))=0=\Psidim(\Gproj(\Lambda)).$$
\end{itemize}
\end{remark}

In the following result, we need the definition of the small Gorenstein global dimension of an artin 
algebra $\Lambda.$ This is defined as follows 
$$\Ggldim(\Lambda):=\Gpd\,(\modu\,(\Lambda)).$$
Note that $\Gfindim(\Lambda^{op})=\Gid\,(\modu\,(\Lambda)).$ We recall that an artin 
algebra $\Lambda$ is Gorenstein if the injective dimensions $\id\,({}_\Lambda\Lambda)$ and 
$\id\,(\Lambda_\Lambda)$ are finite.

\begin{teo}\label{GorAl} Let $\Lambda$ be an artin Gorenstein algebra. Consider the additive 
subfunctors $F:=F_{\Gproj(\Lambda)}$ and 
$G:=F^{\Ginj(\Lambda)}.$ Then, the following statements hold true.
\begin{itemize}
\item[(a)] For any $M,N\in\modu\,(\Lambda)$ there is an 
isomorphism
$$\Ext^i_{\Gproj(\Lambda)}(M,N)\simeq\Ext^i_{\Ginj(\Lambda)}(M,N),$$
which is functorial in both variables.
\item[(b)] $\pd_G(M)=\pd_F(M)=\Gpd(M)$ for any $M\in\C.$
\item[(c)] $\id_F(M)=\id_G(M)=\Gid(M)$ for any $M\in\C.$
\item[(d)] $\id\,(\Lambda_\Lambda)=\Ggldim\,(\Lambda^{op})=\Ggldim\,(\Lambda)=\id\,({}_\Lambda\Lambda).$
\item[(e)] $\Phi_{\Gproj(\Lambda)}\mathrm{dim}(\Lambda)=\PsiGpdim(\Lambda)=\Gid(\Gproj(\Lambda))=\Gpd(\Ginj(\Lambda))=\\=
\Phi^{\Ginj(\Lambda)}\mathrm{dim}(\Lambda)=\Psi^{\Ginj(\Lambda)}\mathrm{dim}(\Lambda)=\id\,({}_\Lambda\Lambda).$
\end{itemize}
\end{teo}
\begin{dem} Note that \cite[Remark 11.5.10]{EJ} say us that $\GProj^{<\infty}(\Lambda)=\Mod\,(\Lambda).$ Then, by Remark \ref{basicasdefi} (c), we conclude that 
$\Gproj^{<\infty}(\Lambda)=\modu\,(\Lambda).$ Using now 
the standard duality functor $D:\modu\,(\Lambda)\to \modu\,(\Lambda^{op}),$ we also have that 
$\Ginj^{<\infty}(\Lambda)=\modu\,(\Lambda).$ Finally, it can be checked, using Remark 
\ref{basicosG} and \cite{Holm2},
that the pairs $(\Gproj(\Lambda),\proj(\Lambda))$ and $(\inj(\Lambda),\Ginj(\Lambda))$ are homologically compatible. Therefore, the result follows from Theorem \ref{Sbalance1}, since $\pd\,(\inj(\Lambda))=\pd\,(D(\Lambda_\Lambda))=\id\,(\Lambda_\Lambda).$
\end{dem}

\begin{teo}\label{T1Sbalance2} Let $\Lambda$ be an artin algebra and $\mathcal{U}$ a resolving, coresolving and functorially finite class in $\modu\,(\Lambda).$ Consider $\X:={}^{\perp_1}\mathcal{U},$  $\Y:=\mathcal{U}^{\perp_1}$  and the functors $F:=F_\X$ and 
$G:=F^\Y.$ If $\X^\wedge\subseteq \Gproj^{<\infty}(\Lambda)$ and $\Y^\vee\subseteq \Ginj^{<\infty}(\Lambda),$
the following statements hold true.
\begin{itemize}
\item[(a)] $\Gid\,(\Y^\vee)=\fid_G(\Lambda)=\fpd_F(\Lambda)=\Gpd\,(\X^\wedge).$
\item[(b)] $\fpd_F(\Lambda)\leq\Phi_{\X}\mathrm{dim}(\Lambda)\leq\Psixdim(\Lambda)\leq\pd_\X\,(\mathcal{U}).$
\item[(c)] $\fid_G(\Lambda)\leq\Phi^{\Y}\mathrm{dim}(\Lambda)\leq\Psi^{\Y}\mathrm{dim}(\Lambda)\leq\id_\Y\,(\mathcal{U}).$
\end{itemize}
\end{teo}
\begin{dem} Note that $\modu\,(\Lambda)$ is an skeletally small category with enough injectives and projectives.  Moreover,  From Remark \ref{Rkbalance2}, we get that the complete hereditary cotorsion pairs $(\X,\mathcal{U})$ and 
$(\mathcal{U},\Y)$ are homologically compatible. Then, by Theorem \ref{Sbalance2} (d1),
it follows that 
$$\fpd_F(\Lambda)\leq\Phi_{\X}\mathrm{dim}(\Lambda)\leq\Psixdim(\Lambda)\leq\coresdim_{\mathcal{U}^{\perp_1}}(\X).$$
Now, by using \cite[Lemma 3.3]{MS}, we get that $\coresdim_{\mathcal{U}^{\perp_1}}(M)\leq\id_{\mathcal{U}}(M)$ for any $M\in\modu\,(\Lambda).$ Hence, we have that 
$$\coresdim_{\mathcal{U}^{\perp_1}}(\X)\leq\id_{\mathcal{U}}(\X)=
\pd_\X(\mathcal{U}),$$
proving  (b). Furthermore, from Theorem  \ref{Sbalance2} (c), we get  
$$\fpd_F(\Lambda)=\pd_{\proj\,(\Lambda)}(\X^\wedge)=\Gpd(\X^\wedge),$$
where the last equality follows from \cite[Theorem 2.20]{Holm2}. Finally, in a similar way, it can shown that (c)  and 
$\Gid\,(\Y^\vee)=\fid_G(\Lambda)=\fpd_F(\Lambda)$ hold true.
\end{dem}

\vspace{1cm}

{\sc Acknowledgments.} The authors thanks  the Project PAPIIT-Universidad Nacional Aut\'onoma de M\'exico IN102914. This work has been partially supported by  project MathAmSud-RepHomol.

\vskip3mm \noindent Marcelo Lanzilotta:\\ Instituto de Matem\'atica y Estad\'{i}stica ``Rafael Laguardia",\\
J. Herrera y Reissig 565, Facultad de Ingenier\'{i}a, Universidad de la Rep\'ublica.\\ CP 11300, Montevideo, URUGUAY.\\
{\tt marclan@fing.edu.uy}

\vskip3mm \noindent Octavio Mendoza Hern\'andez:\\ Instituto de Matem\'aticas, Universidad Nacional Aut\'onoma de M\'exico.\\
Circuito Exterior, Ciudad Universitaria, C.P. 04510, Ciudad de M\'exico,  M\'EXICO.\\
{\tt omendoza@matem.unam.mx}

\end{document}